\documentclass[a4paper]{article}

\usepackage{amssymb,amsfonts,amsthm,amsmath} 
\usepackage{fullpage}
\usepackage{wasysym}

\usepackage{graphicx}
\usepackage[usenames,dvipsnames,svgnames,table]{xcolor}
\usepackage[colorlinks=true, pdfstartview=FitV,linkcolor=ForestGreen,citecolor=ForestGreen, urlcolor=blue]{hyperref}

\usepackage{esint} 
\usepackage{enumitem}

\usepackage[utf8]{inputenc}

\newcommand{\R}{\mathbb{R}}
\newcommand{\Sd}{S^{d-1}}

\newcommand{\eps}{\varepsilon}
\newcommand{\dd}{\, \mathrm{d} }
\newcommand{\dv}{\, \mathrm{d} v}
\newcommand{\ds}{\, \mathrm{d} s}

\newcommand{\dt}{\, \mathrm{d} t}
\newcommand{\dw}{\, \mathrm{d} w}
\newcommand{\dsigma}{\, \mathrm{d} \sigma}
\newcommand{\un}{\mathbf{1}}
\newcommand{\phik}{\varphi_\kappa}
\newcommand{\Phik}{\Phi_\kappa}
\newcommand{\fin}{f_{\mathrm{in}}}
\newcommand{\Cprop}[1]{{C_{\mathrm{prop,#1}}}}

\newcommand{\Cic}{C_{\mathrm{ic}}}

\newcommand{\Cevol}{C_{\text{evol}}}
\newcommand{\Ciia}{C_{\text{(ii)}}}

\newtheorem{thm}{Theorem}[section]
\newtheorem{lem}[thm]{Lemma}
\newtheorem{prop}[thm]{Proposition}

\newtheorem{defi}{Definition}
\theoremstyle{remark}
\newtheorem{rem}{Remark}

\begin{document}

\title{\bf Partial regularity in time for the space-homogeneous Boltzmann equation with very soft potentials}
\author{François Golse, Cyril Imbert and Luis Silvestre}
\date{\today}

\maketitle

\begin{abstract}
  We prove that the set of singular times for weak solutions of the homogeneous Boltzmann equation
  with very soft potentials constructed as in Villani (1998) has Hausdorff dimension at most $\frac{|\gamma+2s|}{2s}$
  with $\gamma \in [-4s,-2s)$  and $s \in (0,1)$.
\end{abstract}

 \tableofcontents

\setcounter{tocdepth}{1}
\tableofcontents

\section{Introduction}

The space-homogeneous Boltzmann equation   posed in $\R^d$, with $d \ge 2$, is
\[ \partial_t f = Q (f,f) \]
where $f= f(t,v) \ge 0$ and $v \in \R^d$, $t>0$ and the collision operator $Q(f,f)$ is defined as,
\[ Q(f,f)(v) = \iint_{\Sd \times \R^d} (f(w')f(v') -f (w) f(v) ) B(w-v,\sigma) \dsigma \dw  \]
with
\[ v' = \frac{v+w}2 + \frac{|v-w|}2 \sigma \quad \text{ and } w' = \frac{v+w}2 - \frac{|v-w|}2 \sigma \]
and the function $B$ is defined by,
\[B(z,\sigma) = |z|^\gamma b (\cos \theta) \quad \text{ where } \quad \cos \theta = \frac{z}{|z|}\cdot \sigma.\]
We are interested in the non cut-off case, \textit{i.e.} we assume that $\gamma > -d$ and that the function $b$ satisfies for some $s \in (0,1)$,
\begin{equation}
  \label{e:b}
  b(\cos \theta) \simeq \frac{1}{|\sin (\theta/2)|^{(d-1)+2s}} \quad \text{ where } \quad \sin (\theta/2) = \frac{v'-v}{|v'-v|} \cdot \sigma.
\end{equation}
The Boltzmann equation is supplemented with an initial condition $f(0,v) = \fin (v)$ for $\fin \ge 0$ measurable with finite mass, entropy and energy. Let $m_0,M_0,E_0,H_0$ be positive constants such that,
\begin{equation}
  \label{e:hydro}
  0 < m_0 \le \int_{\R^d} \fin(v) \dv \le M_0 \quad \text{ and } \quad \int_{\R^d} \fin(v) |v|^2 \dv \le E_0
  \quad \text{ and } \quad \int_{\R^d} \fin \ln \fin (v) \dv \le H_0.
\end{equation}

In the case of hard or moderately soft potentials, \textit{i.e.} $\gamma + 2s \ge 0$, it is known that global solutions  smooth solutions exist \cite{MR4433077}.
This work is concerned with the case of very soft potentials: $\gamma +2s < 0$. It is not known in this case wether solutions are smooth or not. We aim at estimating how large is the set of times around which weak solutions are not essentially bounded. This set will be referred to as the singular set. Before defining it rigourously, we first recall the definition of weak solutions used in this paper.

\subsection{Weak solutions}

We first recall the notion of H-solutions proposed by C. Villani \cite{MR1650006} for the space-homogeneous Boltzmann (and Landau) equation(s) and then introduce a notion of suitable weak solutions in the spirit of L.~Caffarelli, R. Kohn and L. Nirenberg \cite{MR673830} for the Navier-Stokes equations. 

\paragraph{H-solutions.}
 C. Villani introduced in \cite{MR1650006} a new class of weak solutions for the Boltzmann and the Landau equations in $\R^3$, encompassing the very soft potential case. The definition relies on the a priori estimate given by the entropy dissipation of solutions along the flow. These solutions satisfy the conservation of mass, momentum and energy,
\[ \forall t>0, \qquad \int_{\R^3} f(t,v) \left( \begin{array}{c} 1 \\ v \\ |v|^2 \end{array} \right) \dv = \int_{\R^3} \fin (v) \left(\begin{array}{c}1 \\ v \\ |v|^2 \end{array} \right) \dv, \]
together with
\[ \forall t>0, \qquad \int_{\R^3} f \ln f (t,v) \dv + \int_0^t D (f(s,\cdot)) \ds \le \int_{\R^3} \fin \ln \fin (v) \dv, \]
where the entropy dissipation $D(f)$ is defined by
\[ D (f) = \iiint_{\R^d \times \R^d \times \Sd} (f(w')f(v') -f(w)f(v)) \ln \frac{f(w')f(v')}{f(w)f(v)} B(v-w, \sigma) \dv \dw \dsigma  \ge 0. \]
Using the fact that $(x-y) \ln (x/y) \ge 4 |\sqrt{x}-\sqrt{y}|^2$, this implies in particular that
\[ \sqrt{b(\cos \theta)} \frac{\sqrt{F(v',w')} -\sqrt{F(v,w)}}{|v-w|} \in L^2 ((0,T) \times \R^d \times \R^d \times \Sd) \]
with $F(v,w) = f(v) f(w) |v-w|^{\gamma +2}$. 

J. Chaker and the third author recently proved \cite[Corollary~0.4]{chaker2022entropy} that these H-solutions
are in fact weak solutions.  

\paragraph{Suitable weak solutions.}
In order to present the notion of suitable weak solutions used in this paper,
we recall that the collision operator can be written as follows \cite{MR3551261},
\[ Q (f,f)(v) = \int_{\R^d} (f(v') K (v,v') - f(v) K (v',v)) \dv'\]
where the kernel $K \ge 0$ depends on $f$,
\begin{equation}
  \label{e:kernel}
  K (v,v') = 2^{d-1} |v'-v|^{-1} \int_{w \perp v'-v} f(v+w) r^{-d+2} B(r,\sigma) \dw
\end{equation}
with $r = \sqrt{|v'-v|^2+|w|^2}$ and $\sigma = \frac{v-v'-w}{|v-v'-w|}$.

We then compute (formally) for any convex function $\varphi$,
\begin{align*}
  \frac{d}{dt} \int_{\R^d} \varphi (f(t,v)) \dv = &  \int_{\R^d} \varphi' (f) Q (f,f) \dv \\
  =& \iint_{\R^d \times \R^d} \varphi'(f) \bigg(f' K(v,v') - f K(v',v) \bigg)  \dv \dv' \\
  \intertext{use now the function $d_\varphi (r,\rho) = \varphi (\rho) - \varphi (r) - \varphi'(r) (\rho-r)$
  in the first term and make the change of variables $(v,v') \mapsto (v',v)$ in the second term,}
  = &  - \iint_{\R^d \times \R^d} d_\varphi (f,f') K (v,v') \dv \dv' \\
    & + \iint_{\R^d \times \R^d} \bigg(\varphi (f') -\varphi (f)  + \varphi' (f) f - \varphi'(f') f'\bigg) K (v,v') \dv \dv' \\
  = &  - \iint_{\R^d \times \R^d} d_\varphi (f,f') K (v,v') \dv \dv' \\
    & + \int_{\R^d} \Phi (f) \left\{ \int_{\R^d} \bigg(K (v,v') - K (v',v) \bigg) \dv'\right\} \dv
\end{align*}
with $\Phi (f) = (\varphi' (f) f -\varphi (f))$. 
The classical cancellation lemma (see Lemma~\ref{l:cancel} in the next section) allows to compute the last term
in parentheses and get $f \ast_v |\cdot|^\gamma$ (up to some positive constant $c_c$).
The previous computation  suggests to include this family of inequalities associated with non-decreasing convex functions. 

This is reminiscent of the notion of entropy solutions for scalar conservation laws \cite{MR0267257}. It is known that for such a notion of weak solutions for scalar conservation laws, it is sufficient to consider so-called Kruzhkhov entropies $|f-a|$, or simply $(f-a)_+$,  for all $a>0$. Indeed, a general $C^2$ convex function $\varphi$ can be written as follows,
\begin{equation}
  \label{e:convex-rep}
  \forall r>0, \quad \varphi (r) = \varphi(0)+\varphi' (0)r + \int_0^{+\infty} \varphi''(a) (r-a)_+ \dd a .
\end{equation}

In the case of the homogeneous Boltzmann equation, we use non-decreasing convex functions $\varphi$ such that $\varphi (0) = \varphi'(0)=0$. We thus consider $\varphi_a (r) = (r-a)_+$ for all $a>0$. In this case,
\[
  d_{\varphi_a} (f,f') =  \begin{cases} (f'-a)_+ &\text{ if } f \le a \\ (a-f')_+ &\text{ if } f > a \end{cases} 
\]
and $\Phi_a (f) = a \un_{\{ f > a \}}$.
\begin{defi}[Suitable weak solutions] \label{d:suitable}
Let $f\colon (0,T) \times \R^d$ be an $H$-solution of the homogeneous Boltzmann equation. 
It is a \emph{suitable weak solution} if for any $a>0$, we have
\begin{equation}
  \label{e:generalentropy}
  \frac{d}{dt} \int_{\R^d} \varphi_a (f(t,v)) \dv + \int_{\R^d \times \R^d} d_{\varphi_a} (f,f') K (v,v') \dv \dv' \le c_c \int_{\R^d} \Phi_a (f) (f \ast_v |\cdot|^\gamma) \dv
\end{equation}
in the sense of distributions in $(0,T)$, where $\Phi_a (f) = a \un_{\{ f > a \}}$ and   \( d_{\varphi_a} (r,\rho) = \varphi_a (\rho) - \varphi_a (r) - \un_{\{ r>a\}} (\rho-r) \ge 0 \)  and $c_c$ only depends on the function $b$. 
\end{defi}
\begin{rem}
If $f$ is smooth, then the inequality \eqref{e:generalentropy} is in fact an equality. 
\end{rem}
\begin{rem}
  For a general convex function $\varphi$,
  the function $d_\varphi$ is sometimes  referred to as the Bregman distance associated with the convex function $\varphi$ \cite{MR0215617}. 
\end{rem}
\begin{rem}
  If $\varphi$ is twice differentiable, then the function $\Phi$ satisfies $\Phi'(r) = r \varphi'' (r)$. In particular, it is easy to see in this case that $\Phi$ is non-decreasing. Such a property
  holds true for the $\varphi$'s considered in Definition~\ref{d:suitable}. Moreover, since $\varphi (0)=0$,  we have $\Phi (0)=0$ and $\Phi \ge 0$.
\end{rem}
\begin{rem}
  The constant $c_c$ comes from the cancellation lemma (see Lemma~\ref{l:cancel}). 
\end{rem}

\paragraph{Singular times.}
As announced above, we aim at estimating the size of the set of singular times. 
\begin{defi}[Singular times]\label{defi:singular}
Let $f\colon (0,T) \times \R^d$ be a suitable weak solution of the homogeneous Boltzmann equation with initial condition $f(0,v)=\fin (v)$.
A time $t \in (0,T)$ is \emph{regular} if $f$ is bounded in $(t-\delta,t] \times \R^d$ for some $\delta >0$. A time $t \in (0,T)$ is \emph{singular} if it is not regular. 
\end{defi}

\subsection{Main result}

The main result of this article is an estimate of the size of the singular set, that is to say of the set of all singular times. 
\begin{thm}[Partial regularity] \label{t:main}
  Let $f$ be a suitable weak solution of the homogeneous Boltzmann equation with $\gamma \in [-4s,-2s)$ associated with an initial data $\fin$ satisfying \eqref{e:hydro}.
  The Hausdorff dimension of the set of its singular times is at most $|\gamma+2s|/(2s)$. 
\end{thm}
\begin{rem}
It is not known if bounded weak solutions are smooth or not. We think that they are but proving it would require a lot of extra work. 
\end{rem}
\begin{rem}
The case $\gamma \ge -2s$ is contained in the conditional regularity estimate from \cite{MR4433077}. 
\end{rem}
\begin{rem}
In dimension $d=3$, if $\gamma = -3$ and $s \to 1$, the Hausdorff dimension goes to $1/2$. This is consistent with the result for the Landau equation with a Coulomb potential contained in \cite{MR4517683}.  
\end{rem}

Even if the nature of the main result is the same as the one for the Landau-Coulomb equation \cite{MR4517683}, the strategy of proof is different. In \cite{MR4517683}, a De Giorgi method was used to prove
that a solution is bounded in a time interval as soon as the entropy integrated over this interval is sufficiently small. This lemma was supplemented with an iteration argument to ensure that after scaling, this criterion is satisfied around times where the entropy dissipation is small.

The proof of Theorem~\ref{t:main} relies on the propagation of the $L^p$-norm of a solution. At a formal level, it is possible to ensure that if the $L^p$-norm is finite at time $t_0$, it will be controlled
for some (small) time. This is obtained by proving that the $L^p$-norm satisfies a Riccati equation. In order to make this formal argument rigourous, we consider an approximate $L^p$-norm. It is defined by a
sublinear convex function. More precisely, it is obtained by truncating the derivative of $r^p$ at a threshold $\kappa$. It is then necessary to derive an approximate Riccati equation. In order to do so,
we use the computation presented above about the evolution with time of the quantity $\int_{\R^d} \varphi (f(t,v)) \dv$.

We then use the conditional boundedness result obtained by the third author in \cite{MR3551261}. It asserts (more or less) that a solution is bounded in a time interval as soon as its $L^\infty_t L^p_v$-norm is small for $p > 3/2$. Since we work with suitable weak solutions, it is necessary to modify the proof of \cite{MR3551261}. We do so by using ideas coming from \cite{chaker2022entropy} and \cite{ouyang2023conditional}. 

While we were working on this project, we learned that the solutions of the space-homogeneous Landau equation with Coulomb potentials do not blow up. It is a consequence of the fact that the Fisher information decreases along the flow of the equation, which was recently proved by Nestor Guillen and the third author in \cite{guillen2023landau}.
In view of the result for the Landau equation, and recalling that this equation is obtained as a limit of the Boltzmann equation for a strong angular singularity, it is conceivable that the Fisher information also decreases for the space-homogeneous Boltzmann equation with very soft potentials. Such a result is still an open problem, and it would imply the global regularity of the solutions. For the Landau equation, there is a condition on the interaction potential for the monotonicity of the Fisher information to hold. We would expect that if one could prove a corresponding result in the case of the Boltzmann equation, there would also be a similar condition on the collision kernel. The exact nature of this condition is hard to predict at the moment. Because of this, it is possible that the partial regularity obtained in this paper may have some applicability to the cases in which the Fisher information is not monotone, if such a result was ever obtained.

\subsection{Review of literature}

We review previous works related to the Boltzmann equation without cut-off. The literature related to this topic is very rich and the few following paragraphs are by no means exhaustive.
The reader is referred to the references contained in the  articles cited below for a more complete picture of the available results.

C. Villani \cite{MR1650006} constructed weak solutions of the homogeneous Boltzmann equation in the non-cut off case for
moderately soft potentials. As far as very soft potentials are concerned, he introduced a new class of weak solutions, based on the entropy dissipation estimate. The coercivity property of the collision operator of the
Boltzmann equation was further investigated in \cite{ADVW}. This influential article also contains a tool that is nowadays classical, the cancellation lemma (see Lemma~\ref{l:cancel} below). Ten years later, 
P. Gressman and R. Strain \cite{MR2784329} constructed global solutions of the Boltzmann equation close to equilibrium in the inhomogeneous setting. Independently, R. Alexandre, Y. Morimoto, S. Ukai,  C.-J. Xu and T. Yang \cite{AMUXY} obtained conditional regularity results in the inhomogeneous case too. 

The third author introduced in \cite{MR3551261} new ideas to study the conditional regularity of solutions of the Boltzmann equation in the inhomogeneous setting without cut-off and for moderately soft potentials. He observed that if  mass, energy and entropy are under control then  the ellipticity of the collision operator is strong enough to get a pointwise estimate of the solution (see Theorem~\ref{t:prodi-serrin}). This observation was further developped with the second author in \cite{MR4049224}. In this work,  a class of kinetic equations with integral diffusion is introduced and a modulus of continuity is derived in the spirit of the classical works by De Giorgi and Nash. The final conditional regularity result in this setting appears in \cite{MR4433077}. It takes the form of a  global regularity estimate satisfied by solutions as long as the hydrodynamical bounds hold true. In the homogeneous setting, this implies (at least formally) that there exists a global smooth solution for moderately soft potentials. 

Together with M. Gualdani and A.  Vasseur \cite{MR4517683}, the first and second authors studied partial regularity in time of solutions of the Landau equation in the case of a Coulomb potential.
They used De Giorgi method in order to prove that the set of singular times is at most of Hausdorff dimension $1/2$. 

J. Chaker and the third author obtained in \cite{chaker2022entropy} a coercivity estimate that can be seen as the counterpart of Desvillettes's entropy dissipation inequality \cite{MR3369941}. In particular, it implies that H-solutions are weak for very soft potentials. 
Z. Ouyang and the third author extended in \cite{ouyang2023conditional} the pointwise bound from \cite{MR3551261} to the case where $x$ lies in a domain and the equation is supplemented with  boundary conditions (in-flow, bounce-back, specular-reflection and diffuse-reflection). 

\paragraph{Organisation of the article.}
In Section~\ref{s:prelim}, we gather known facts and technical results that will be used in the remainder of the article.
We recall the coercivity properties of the collision operator, associated with its non-degeneracy, as well as
the propagation of moments by the Boltzmann equation in the very soft potential case. We also state and prove various
interpolation inequalities taylored for our needs. Section~\ref{s:coercivity} contains coercivity estimates associated with
the dissipation of approximate Lebesgue norms.  Section~\ref{s:propagation} is devoted to the study of the evolution of approximate $L^p$-norms. In Section~\ref{s:criterion}, the conditional boundedness result obtained by the third
author in \cite{MR3551261} is extended to suitable weak solutions. In the final section~\ref{s:partial}, we prove our main
result, Theorem~\ref{t:main}. Suitable weak solutions are constructed in an appendix, see Section~\ref{s:construction}.

\paragraph{Notation.}
The unit ball of $\R^d$ centered at the origin is denoted by $B_1$ and $\Sd$ denotes the unit sphere in $\R^d$.  The ball of radius $r>0$ centered at the origin is denoted by $B_r$. 

For $p \in [1,\infty]$, the  space $L^p$ is the usual Lebesgue space in $\R^d$.
For $p \in [1,\infty]$ and $k \in \R$,  $L^p_k$ denotes the following weighted Lebesgue space,
\[ L^p_k = \{ f \colon \R^d \to \R, \text{measurable s.t.} \|f\|_{L^p_k} < +\infty \} \]
with $\langle v \rangle = (1+|v|^2)^{\frac12}$ and
\[ \|f\|_{L^p_k} = \left(\int_{\R^d} f^p (v) \langle v \rangle^{kp} \dv \right)^{\frac1p}.\]

\section{Preliminaries}
\label{s:prelim}

\subsection{About the Bregman distance}

\begin{lem}[Bregman distance]\label{l:breg}
The function $d_{\varphi} (r,\rho)$ is non-increasing with respect to $\rho$ in $(-\infty,r]$
and non-decreasing with respect to $r$ in $[\rho,+\infty)$. 
\end{lem}
\begin{proof}
  It is enough to compute partial derivatives. For $\rho \le r$, we have,
  \[
    \frac{\partial d_\varphi}{\partial \rho}(r,\rho)  = \varphi' (\rho) - \varphi' (r) \le 0 \quad \text{ and } \quad
    \frac{\partial d_\varphi}{\partial r}(r,\rho) = - \varphi'' (r) (\rho -r) \ge 0. \qedhere
  \]
\end{proof}

\subsection{Collision kernel}

We state here the cancellation lemma in the form contained in \cite{MR3551261}. 
\begin{lem}[Cancellation lemma -- {\cite[Lemma~1]{ADVW}}] \label{l:cancel}
Let $K$ be given by \eqref{e:kernel}. Then,
  \[\int_{\R^d} \bigg(K (v,v') - K (v',v) \bigg) \dv' =  f \ast_v \mathcal{R}\] with
\[ \mathcal{R} (r) = \int_{\Sd} \left(\left(\frac{1-\sigma \cdot e}2\right)^{-d/2} B \left( \frac{\sqrt{2} r}{\sqrt{1-\sigma \cdot e}} , \sigma \cdot e \right) - B(r, \sigma \cdot e) \right) \dsigma \]
for an arbitrary unit vector $e \in \Sd$ (the formula is independent of the choice of $e$).

In particular, when $B(r,\cos \theta) = r^\gamma b(\cos \theta)$, we have $\mathcal{R}(r) =c_c r^\gamma$  for some constant $c_c>0$ only depending on the function $b$. 
\end{lem}

\begin{thm}[Entropy dissipation estimate -- {\cite[Corollary~0.3]{chaker2022entropy}}] \label{t:entropy}
  Let $\gamma \in (-d,2]$ and $s \in (0,1)$. If $f_0$ has finite entropy, then
  any H-solution $f$ of the Boltzmann equation satisfies
  \[ \|f \|_{L^1 ([0,T],L^{p_0}_{k_0})} \le C_0,  \]
  where $p_0$ is such that $\frac{1}{p_0} = 1 - \frac{2s}d$ and $k_0 = \gamma +2s - 2s/d $.
  The constant $C_0$ only depends on mass, energy and entropy of the initial data. 
\end{thm}

\subsection{Non-degeneracy}

We first recall the existence of a non-degeneracy cone ensuring that, for $v$ fixed, the
kernel $K(v,v')$ is comparable (up to some weight) to the kernel of the fractional Laplacian
for $v'-v$ in a cone. 
\begin{lem}[Non-degeneracy cone -- {\cite[Lemma~7.1]{MR3551261}}] \label{l:cone}
  For every $v \in \R^d$, there exists a symmetric cone $C(v)$\footnote{\textit{i.e.} if $w \in C(v)$ and $\ell \in \R$, then $v + \ell (w-v) \in C(v)$ }
  such that
  \[ K (v,v') \ge \lambda_0 \langle v \rangle^{\gamma+2s+1} |v'-v|^{-d-2s} \quad \text{ for every } v' \in C(v)\]
  and $|C(v) \cap B_1 (v)| \ge \lambda_0 \langle v \rangle^{-1}$.
  The constant $\lambda_0 >0$ only depends on dimension $d$ and hydrodynamical bounds from \eqref{e:hydro}. 
\end{lem}
\begin{rem}\label{r:tr}
  Remark that we can also ensure that $|C(v) \cap B_3 (v) \setminus B_2 (v)| \ge \lambda_0 \langle v \rangle^{-1}$.
\end{rem}

This lemma is a consequence of the following classical lemma, see \cite{MR3551261}.
\begin{lem}[Pointwise lower bound from hydrodynamical bounds -- {\cite[Lemma~4.6]{MR3551261}}]\label{l:pw-lower-bound}
  Let $f \colon \R^d \to [0,+\infty)$ be measurable and such that \eqref{e:hydro} holds true. 
  There then exist constants $\ell_0\in (0,1)$, $a_0 >0$ and $R_0>0$ only depending on dimension and $m_0,M_0,E_0,H_0$
  such that
  \[ \left|\{ f \ge \ell_0 \} \cap B_{R_0}  \right| \ge a_0 >0.\]
\end{lem}

In this work, we need a lower bound on the kernel $K(v,v')$ for $v'$ fixed. This is distinct
from the previous non-degeneracy result. We consider,
\begin{equation}
  \label{e:N}
  S (v') = \{ v \in \R^d : K(v,v') \ge \mu_0 \langle v' \rangle^{\gamma +2s +1} |v'-v|^{-d-2s} \}.
\end{equation}
This set $S(v')$ has positive measure in a neighborhood of the sphere centered at $v'/2$ of radius $|v'|/2$. The neighborhood becomes thiner  as $|v'| \to \infty$.
\begin{lem}[Non-degeneracy set]\label{l:sphere}
There exist two positive constants $\eps_0$ and $\mu_0$, only depending on the hydrodynamical bounds from \eqref{e:hydro} and dimension,  such that for $v' \in \R^d$ and $R>0$ satisfying, $R \le \eps_0 \langle  v' \rangle$, we have $|S(v') \cap B_R (v')| \ge \mu_0 \frac{R^d}{\langle v' \rangle}$
  
\end{lem}
\begin{proof}
We distinguish the cases $|v'| \ge 4R_0$ and $|v'| \le 4R_0$. The first one is more complicated
since we have to get a lower bound with the proper decay rate for large values of $v'$. 
\bigskip

\noindent \textsc{The case of large values of $|v'|$.}
Let us assume $|v'| \ge 4R_0$. 
  We use \cite[Corollary~4.2]{MR3551261} in order to write for some constant $c_b$ only depending on $b$,
\begin{align*}
  K(v,v') &\ge c_b \left\{ \int_{w \perp v'-v} f(v+w)|w|^{\gamma+2s+1} \dw\right\} |v-v'|^{-d-2s} \\
  & \ge c_b \ell_0 \left\{ \int_{w \perp v'-v} \un_{A_0}(v+w) |w|^{\gamma+2s+1} \dw\right\} |v-v'|^{-d-2s}
\end{align*}
where $A_0 = \{ f \ge \ell_0 \} \cap B_{R_0}$. We know from Lemma~\ref{l:pw-lower-bound} that $|A_0 | \ge a_0 >0$. 
We consider
\[ k(v,v') = \int_{w \perp v'-v} \un_{A_0}(v+w) |w|^{\gamma+2s+1} \dw.\]
Remark that
\[S(v') \supset \{ v \in \R^d : k(v,v') \ge \mu_0 \langle v' \rangle^{\gamma+2s+1} \}.\]
We are thus looking for a lower bound on
\[ \left|\{ v \in B_R (v') : k(v,v') \ge \mu_0 \langle v' \rangle^{\gamma+2s+1} \} \right|.\]

Let $v \in B_R (v')$ and $w$ such that $v+w \in A_0 \subset B_{R_0}$ and $R \le \eps_0 \langle v' \rangle$.
We have
\[ R \le \eps_0 (1+|v'|) \le \eps_0 ((4R_0)^{-1}+1) |v'|.\]
We thus pick $\eps_0>0$ such that $\eps_0 ((4R_0)^{-1}+1) \le 1/4$. With such a choice, we have
\[ R \le |v'|/4 \quad \text{ and } |v'|/2 \ge R + R_0.\]

Writing $w = -v' + (v'-v) + (v+w)$, we have
\begin{equation}
  \label{e:ineq-up}
\begin{cases}  |w|  \le |v'| + R+R_0 \le 2 \langle v' \rangle\\
  |w|  \ge |v'| - R - R_0 \ge |v'|/2 \ge \alpha_0 \langle v' \rangle
\end{cases}
\end{equation}
for some $\alpha_0>0$ only depending on $R_0$ and
\[  \begin{cases}
    |w+v-v'|  \le |v'| + R \le 2 \langle v' \rangle\\
  |w+v-v'|^2-|v-v'|^2  = |w|^2 \ge  \alpha_0^2 \langle v' \rangle^2
\end{cases}
\]
(we used that $w \perp v'-v$).
In particular,
\begin{equation}
  \label{e:ineq-low}
  \begin{cases}
    |w+v-v'|  \le |v'| + R_0 \le 2 \langle v' \rangle\\
  |w+v-v'|-|v-v'|\ge  (\alpha_0^2/3) \langle v' \rangle
\end{cases}
\end{equation}
(we used that $|w+v-v'|+|v-v'|\le 2 \langle v' \rangle + R \le 3 \langle v' \rangle$ to get the second inequality). 
\bigskip

\noindent \textsc{Upper bound for $k$.}
We start with getting an upper bound for $k(v,v')$ in $S(v') \cap B_R (v')$.
Estimates~\eqref{e:ineq-up} imply that
\[ |k(v,v') | \le \alpha_1 \langle v'\rangle^{\gamma+2s+1} \int_{w \perp v'-v} \un_{B_{R_0}} (v+w) \dw \le \alpha_1 \omega_{d-1} R_0^{d-1}
  \langle v'\rangle^{\gamma+2s+1} \]
for some $\alpha_1>0$ only depending on $\alpha_0$ and $\gamma,s$ (depending on the sign of $\gamma+2s+1$). 
We thus proved that
\begin{equation}
  \label{e:k-upper}
  |k(v,v')| \le C_k \langle v'\rangle^{\gamma+2s+1} \text{ in } S(v') \cap B_R (v')  
\end{equation}
with $C_k = \alpha_1 \omega_{d-1} R_0^{d-1}$.
\bigskip

\noindent \textsc{Integration over $B_R (v')$.}
We now aim at estimating the integral of $k$ over the ball centered at $v'$. In order to do so, let $\rho \in [0,R]$ and let us integrate $k$ over $\partial B_\rho$,
\begin{align*}
  \int_{\partial B_\rho} k (v' + \sigma,v') \dsigma & \ge \alpha_2 \langle v'\rangle^{\gamma+2s+1} \int_{\partial B_\rho} \left( \int_{w \perp \sigma} \un_{A_0}(v'+\sigma + w) \dw \right) \dsigma \\
\intertext{for $\alpha_2$ only depending on $\alpha_0$ and $\gamma,s$ (we used \eqref{e:ineq-up}), then we use \cite[Lemma A.10]{zbMATH07174139} and get}
 & \ge \alpha_2 \langle v' \rangle^{\gamma+2s+1} \omega_{d-2} \rho^{d-1} \int_{(A_0 -v') \setminus B_\rho} \frac{(|z|^2 -\rho^2)^{\frac{d-3}2}}{|z|^{d-2}} \dd z. 
\end{align*}
We remark that $(A_0 -v') \cap B_\rho = \emptyset$ for $\rho \le R$. Indeed, for $z \in A_0$,
\[ |z-v'| \ge |v'|-R_0  \ge R \ge \rho.\]
For $z \in (A_0-v') \setminus B_\rho$, \eqref{e:ineq-low} yield $|z| \le  2 \langle v' \rangle$ and $|z| - \rho \ge \alpha_0 \langle v' \rangle$; the latter inequality implies $|z|^2  -\rho^2 \ge \alpha_0^2 \langle v' \rangle^2$. Hence,
\[ \frac{(|z|^2 -\rho^2)^{\frac{d-3}2}}{|z|^{d-2}} \ge \frac{\alpha_0^{d-3} \langle v' \rangle^{d-3}}{2^{d-2}\langle v' \rangle^{d-2}} = \frac12 (\alpha_0/2)^{d-3} \langle v'\rangle^{-1}. \]
This implies that
\[
  \int_{\partial B_\rho} k (v,v') \dsigma  \ge (\alpha_2/2) (\alpha_0/2)^{d-3} \omega_{d-2} a_0 \langle v' \rangle^{\gamma+2s+1}  \langle v'\rangle^{-1} = \alpha_3 \langle v'\rangle^{\gamma+2s+1}  \langle v' \rangle^{-1} 
\]
with $\alpha_3= (\alpha_2/2) (\alpha_0/2)^{d-3} a_0 \omega_{d-2} $. We used that $|A_0| \ge a_0$. 
Integrating over $\rho \in [0,R]$, we get for $c_k= \frac1d \alpha_3$, 
\begin{equation}
  \label{e:k-lower}
  \int_{B_R (v')} k (v,v') \dsigma \ge c_k \langle v'\rangle^{\gamma+2s+1} \frac{R^d}{\langle v'\rangle}.
\end{equation}
\bigskip

\noindent \textsc{Estimate of the sub-level set.}
We can now get a lower bound on a sub-level set of $k$ as follows. We use the upper bound of $k$ in $S(v') \cap B_R (v')$ and the lower bound on the integral
over $B_R (v')$ in order to write,
\begin{align*}
  c_k \langle v'\rangle^{\gamma+2s+1}  \frac{R^d}{\langle v'\rangle} & \le \int_{B_R (v')} k(v,v') \dv \\
  & \le C_k \langle v'\rangle^{\gamma+2s+1} | S(v') \cap B_R (v')| + \mu_0 \langle v'\rangle^{\gamma+2s+1} |B_R (v') \cap \{ k >0 \}|.
\end{align*}
We now remark that $k$ is supported in the union of spheres corresponding to diameters $[z_0,v']$ with $z_0 \in A_0$.
In particular, it is contained in the union of such spheres corresponding to $z_0 \in B_{R_0}$. The measure of the
intersection of this larger set with $B_R(v')$ is bounded from above by $R^d/|v'|$, up to some constant $C_0$ only depending
on the radius $R_0$ and dimension $d$.
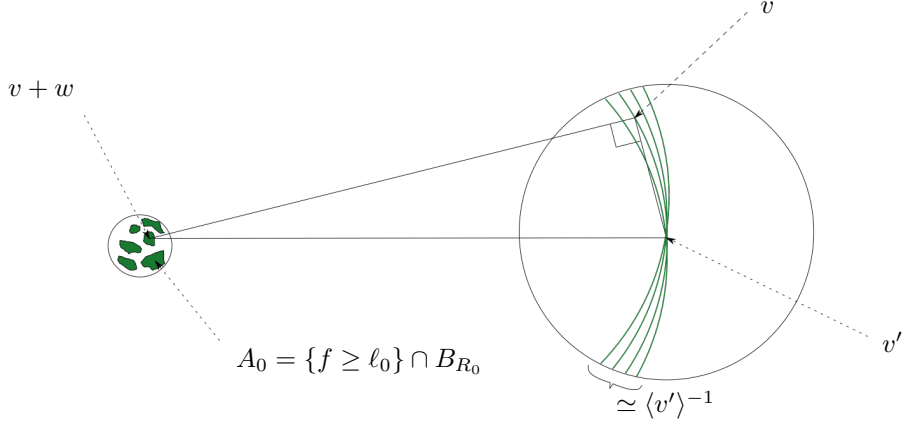
\begin{figure}
     \setlength{\unitlength}{1cm}
\begin{picture}(10,5)
\put(2,0){\includegraphics[height=5cm]{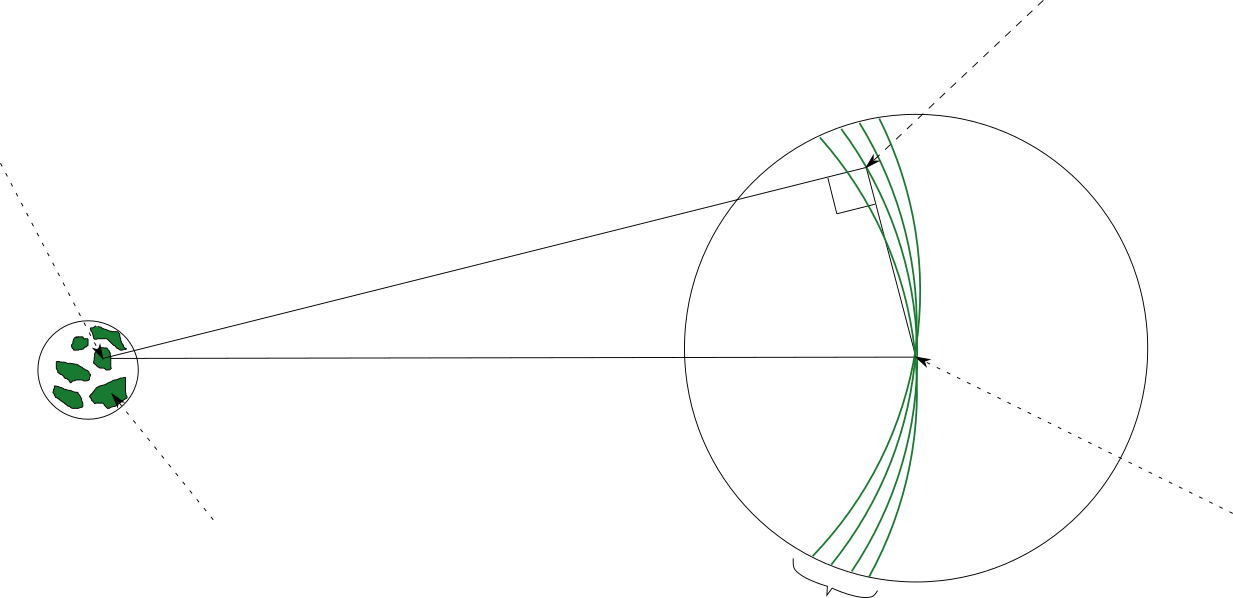}}
       \put(12.5,0.5){$v'$}
       \put(10.9,5){$v$}
       \put(1,3.9){$v+w$}
       \put(4,0.3){$A_0 = \{ f \ge \ell_0\} \cap B_{R_0}$}
       \put(9,-0.3){$\simeq \langle v' \rangle^{-1}$}
     \end{picture}
\caption{The non-degeneracy set $S(v')$ is made of hyperspheres  of diameter $[z_0,v']$ with $z_0 \in A_0$.}
\end{figure}
We finally get,
\[
  c_k  \frac{R^d}{\langle v'\rangle}  \le C_k | S(v') \cap B_R (v')| + \mu_0  C_0 \frac{R^d}{\langle v' \rangle}.
\]
We conclude by choosing $\mu_0$ such that $c_k - \mu_0 C_0 \ge \mu_0 C_k$, for instance $\mu_0 = c_k/(C_0+C_k)$.
\bigskip

\noindent \textsc{The case of small values of $|v'|$.}
We now explain how to treat the case $|v'| \le 4R_0$. Since $R \le \eps_0 \langle v' \rangle$ (for some $\eps_0$ to be chosen),
we see that we only need to get a lower bound on $|S(v') \cap B_R (v')|$ independent of  $v'$. In view of the treatment of the case $|v_0| \ge 4 R_0$,
we need estimates similar to \eqref{e:ineq-up} and \eqref{e:ineq-low}. It is clear that we have the upper bounds for $|w|$ and $|w+v-v'|$. As far as lower
bounds are concerned, we only need a lower bound on $|w|$. Writing $w= -v'+ (v'-v) + (v+w)$, we see that
\[|w| \ge d(A_0, B_R (v')). \]
We consider $\iota_0>0$ such that $|B_{\iota_0}| = \mu_0/2$. In particular, the set $A_0^\ast = A_0 \setminus B_{\iota_0}(v')$ satisfies
\[|A_0^\ast| \ge \mu_0/2\]
and for $R \le \iota_0/2$, we have
\[ d(A_0^\ast, B_R (v')) \ge \iota_0/2.\]
 \begin{figure}[h]
     \setlength{\unitlength}{1cm}
     \begin{picture}(8,5)
\put(4.5,0.5){\includegraphics[height=4cm]{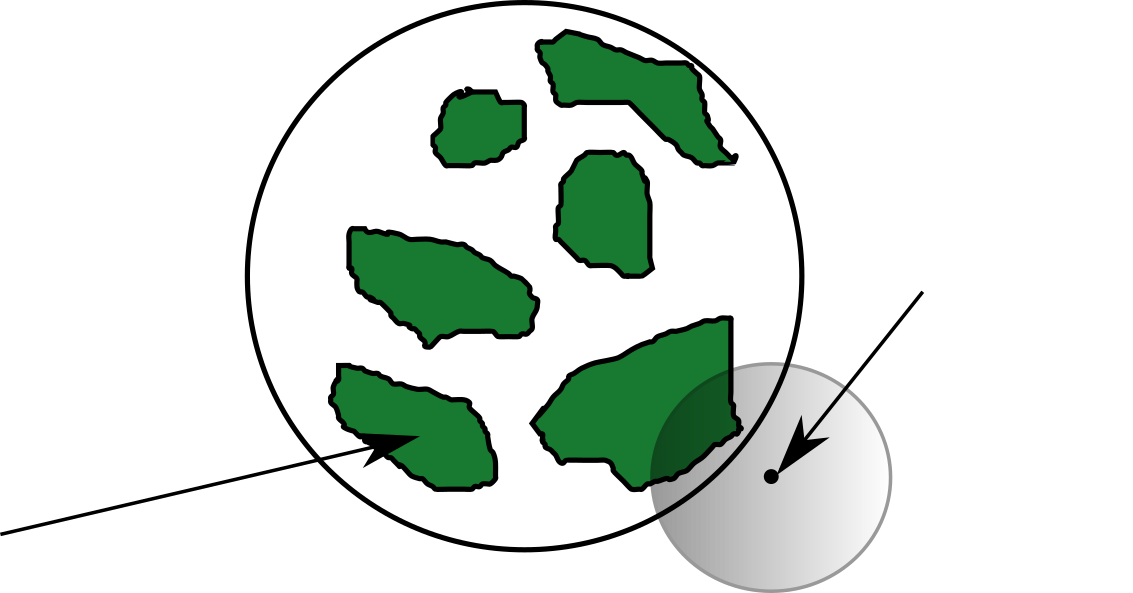}}
       \put(10.8,2.5){$v'$}
       \put(1.2,0.7){$A_0 = \{ f \ge \ell_0\} \cap B_{R_0}$}
      \end{picture}
      \caption{The set $A_0$ where $f$ is bounded from below is represented in green.
        The grey ball is centered at $v'$ and its diameter is $\iota_0$.}
 \end{figure}
Hence, for $v+w \in A_0^\ast$ and $v \in B_R (v')$ and $R \le \iota_0/2$, we have
\[ |w| \ge \iota_0/2.\]
To ensure that $R \le \iota_0$, we pick $\eps_0 \le \frac{\iota_0}{2(1+4 R_0)}$ such that
\[ R \le \eps_0 (1+|v'|) \le \eps_0 (1+4R_0) \le \iota_0/2.\]
Now we can adapt the argument for large values of $|v'|$ by replacing $A_0$ with $A_0^\ast$. Remark that the geometric argument in the last step of the previous case to justify that the support of $k$ in $B_R (v')$ is not too large is not needed for small values of $|v'|$. 
\bigskip

\noindent \textsc{Conclusion.} We choose $\eps_0 = \min \left(\frac{R_0}{1+4R_0}, \frac{\iota_0}{2(1+4 R_0)} \right)$ and
get $|S(v') \cap B_R (v') | \ge \mu_0 R^d \langle v' \rangle^{-1}$ for $R \le \eps_0 \langle v' \rangle$ in all cases.
\end{proof}

\subsection{Propagation of moments}

\begin{thm}[Propagation of moments -- {\cite[Theorem~1 (I)]{MR2546739}}]\label{t:moments}
  Let $f(t=0) = f_0$ be such that $\int_{\R^d} f_0(v) |v|^k \dv < +\infty$ for some $k >0$.
  Let $f \colon (0,+\infty) \times \R^d \to [0,+\infty)$ be an H-solution of the Boltzmann equation
  with $B(z,\cos \theta) = \rho (z) b (\cos \theta)$ such that there exist two constants
  $A^*$ and $A_*$ and $\gamma \in [-4,0)$ such that for all $z \in \R^d$, 
  \[ A_* (1+|z|^2)^{\gamma/2} \le \int_{\Sd} \rho (|z|) b(\cos \theta) \sin^2 (\theta) \dsigma \le A^* |z|^\gamma.\]

  Then for all time $t>0$,
  \[ \int_{\R^d} f(t,v) |v|^k \dv \le \Cprop{k} \]
  where $C_k$ only depends on dimension, $\gamma$, $s$, $\int_{\R^d} f_0(v) |v|^k \dv$ and $H_0 = \int_{\R^d} \fin \ln \fin (v) \dv$. 
\end{thm}

\subsection{Functional inequalities}

We start with an estimate of an integral involving the convolution with $|\cdot|^\gamma$ that naturally appears
after applying the cancellation lemma. 
\begin{lem}[Convolution product] \label{l:convol}
  Let $p \in [1, d/(d+\gamma))$. Then 
  \[\int_{\R^d} g (f \ast_v |\cdot|^\gamma) \dv \le C_{d,\gamma} \| g \|_{L^{q_1/p}} \|f\|_{L^{q_1}}\]
  with $q_1 \in (p,p+1)$ such that $\frac{p+1}{q_1} = 2 + \frac{\gamma}d$. The constant $C_{d,\gamma}$ only depends on dimension and $\gamma$. 
\end{lem}
\begin{rem} \label{r:weak-young}
  The constant $C_{d,\gamma}$ equals $\| |\cdot|^\gamma \|_{L^{-d/\gamma,\infty}}$. 
\end{rem}
\begin{proof}
  It is a consequence of H\"older's inequality and the weak Young's inequality. Indeed, for $p_1>1$ such that
  $\frac1{p_1} + \frac{p}{q_1} = 1$, we have,
  \begin{align*}
    \int_{\R^d} g (f \ast_v |\cdot|^\gamma) \dv & \le \| g\|_{L^{q_1/p}} \| f \ast_v |\cdot|^\gamma \|_{L^{p_1}} \\
    & \le \| g\|_{L^{q_1/p}} \| f \|_{L^{q_1}} \| |\cdot|^\gamma \|_{L^{r,\infty}}
  \end{align*}
  with $\frac1{q_1} + \frac1r = 1 + \frac1{p_1}$ and $r = -d/\gamma$. 
\end{proof}

We continue with an interpolation estimate. 
\begin{lem}[An interpolation estimate] \label{l:interpol}
  Let $p \in \left(\frac{d}{d+\gamma+2s},\frac{d}{d+\gamma} \right)$ 
  and  $q_1 \in (p,p+1)$ such that $\frac{p+1}{q_1} = 2 + \frac{\gamma}d$ and $\frac1{p_0} = 1 - \frac{2s}d$

  Then $p < q_1 < p p_0=q$ and for all $\eps>0$, there exists $C_\eps >0$ such that
  \[ \| f\|_{L^{q_1}}^{p+1} \le \eps \| f^p\|_{L^{p_0}}    + C_\eps  \|f \|_{L^p}^{p \beta}. \]
  The constant $\beta$ is given by the following formula,
  \[
    \beta = 1 + \frac{2s}{p(d+\gamma+2s)-d} . 
  \]
\end{lem}
\begin{proof}
We start with H\"older's inequality with $\frac1{q_1} = \frac{\alpha}p + \frac{1-\alpha}{p p_0}$,
\[
  \|f\|_{L^{q_1}}^{p+1}  \le \|f\|_{L^p}^{\alpha(p+1)} \|f\|_{L^{pp_0}}^{(1-\alpha)(p+1)}.
\]
Then we apply H\"older inequality for products: $ab \le \frac{1}{P} a^P + \frac1{Q}b^Q$ as soon as $\frac1P +\frac1Q =1$.
We apply it with $P= \frac{p}{(1-\alpha)(p+1)}$.

We have to check that $P>1$.  
We first remark that
\[ (1-\alpha) = \frac{1-\frac{p}{q_1}}{1-\frac1{p_0}} = \frac{d}{2s} \left(1 -\frac{p}{q_1} \right).\]
Then
\begin{align*}
  \frac1P & = (1-\alpha) \left(1+\frac1p \right) \\
          & = \frac{d}{2s} \left(1 -\frac{p}{q_1} \right)\left(1+\frac1p \right)\\
          & = \frac{d}{2s} \left( 1 +\frac1p -\frac{p+1}{q_1} \right)\\
  &= \frac{d}{2s} \left( \frac1p -1 - \frac{\gamma}d \right).
\end{align*}
Since $\frac1p < 1 + \frac{\gamma+2s}{d}$, we get $\frac1P<1$. 

We then apply $ab \le \frac{1}{P} a^P + \frac1{Q}b^Q$ with $b=\eps^{-1}\|f\|_{L^p}^{\alpha(p+1)}$ and $a=\eps \|f\|_{L^{pp_0}}^{(1-\alpha)(p+1)}$ and get
\begin{equation}\label{e:no-weight}
\|f\|_{L^{q_1}}^{p+1} \le \frac{\eps^{-Q}}{Q} \|f\|_{L^p}^{\alpha(p+1)Q} + \frac{\eps^p} P \|f\|_{L^{pp_0}}^p.
\end{equation}

We now compute $\alpha (p+1)Q$. In order to do so, we compute
\[ \frac1Q = 1 -\frac1P = 1-(1-\alpha)\left(1+\frac1p \right) = \frac{\alpha (p+1)-1}{p}.\]
Hence,
\[ \alpha (p+1) Q =  \beta p\]
with
\[ \beta = \frac{\alpha(p+1)}{\alpha (p+1)-1} = 1 + \frac1{\alpha (p+1)-1}.\]
We now express $\beta$ in terms of $p,d,\gamma,s$ starting from $\frac1{q_1} = \frac\alpha{p}+\frac{1-\alpha}{pp_0}$.
Recalling that $\frac{p+1}{q_1} = 2 + \frac{\gamma}d$, we get
\begin{align*} 2 + \frac{\gamma}d &= \frac{\alpha(p+1)}p + \frac{(1-\alpha)(p+1)}{pp_0} \\
                                  &= \frac{\alpha(p+1)}p \left(1-\frac1{p_0} \right) + \frac{p+1}{p} \frac1{p_0} \\
                                  &= \frac{\alpha(p+1)}p \frac{2s}d + \frac{p+1}{p} \left(1-\frac{2s}d \right).
\end{align*}
From this equality, we get
\begin{align*}
  \alpha (p+1) &= \frac{2d+\gamma}{2s} p - \frac{d-2s}{2s}  (p+1) \\
& = 1 + \frac{p(d+\gamma+2s) -d}{2s}.
\end{align*}

We rewrite \eqref{e:no-weight} as
\[ \|f\|_{L^{q_1}}^{p+1} \le C_{\bar \eps} \|f\|_{L^p}^{\beta p} + \bar \eps \|f^p\|_{L^{p_0}}\]
where $\bar \eps = \eps /P$ and $C_{\bar \eps}$ is a large constant.
\end{proof}

The previous interpolation estimate is not enough for our needs, weights in $v$ are to be considered.
In order to get such a generalization, we will use H\"older's inequality with weights.
\begin{lem}[H\"older's inequality with weights]\label{l:holder-weight}
  Let $p,q,r \in (1,+\infty)$ and $p < r < q$. Then,
  \[ \| f  \|_{L^r_{k_r}} \le \| f \|_{L^p_{k_p}}^\alpha
    \| f  \|_{L^q_{k_q}}^{1-\alpha} \]
  with $\alpha \in (0,1)$ such that
  \( \frac1{r} = \frac{\alpha}p + \frac{1-\alpha}q\)  and \( k_r = \alpha k_p + (1-\alpha)k_q.\)
\end{lem}
\begin{proof}
  It is enough to apply $\|FG\|_{L^r} \le \|F\|_{L^{\frac{p}\alpha}} \|G\|_{L^{\frac{q}{1-\alpha}}}$ to $F=f^\alpha \langle v \rangle^{\alpha k_p}$
  and $G = f^{1-\alpha} \langle v \rangle^{(1-\alpha)k_q}$. 
\end{proof}
We can now state and prove our interpolation estimate with weights. 
\begin{lem}[An interpolation estimate with weights] \label{l:interpol-weight}
  Let $p \in \left(\frac{d}{d+\gamma+2s},\frac{d}{d+\gamma} \right)$ 
  and  $q_1 \in (p,p+1)$ such that $\frac{p+1}{q_1} = 2 + \frac{\gamma}d$.

  Then $p < q_1 < p p_0=q$ and for  $\eps>0$ small enough and $k_0<0$, there exists $k_1>0$ and $C_\eps >0$ such that
  \[ \| f\|_{L^{q_1}}^{p+1} \le \eps \| f^p\|_{L^{p_0}_{k_0}}
    + C_\eps \left( \|f \|_{L^p}^{p \beta_\eps} +  \|f \|^{r_\eps}_{L^1_{k_\eps}} \right) \]
  for some $r_\eps> 1$ and $k_\eps >0$ (large). The constant $\beta_\eps$ satisfies $\beta_\eps \to \beta$ as $\eps \to 0$ with $\beta>1$  given by the following formula,
  \begin{equation} \label{e:beta}
    \beta = 1 + \frac{2s}{p(d+\gamma+2s)-d} . 
  \end{equation}
  The weight exponent $k_1$ depends on $k_0$, $p,d,s,\gamma$ and $\eps$. 
\end{lem}
\begin{rem}
The exponent $p_1$ appears in Lemma~\ref{l:convol} above.
  The exponent $p_0$ and the weight $k_0$ appear in Lemma~\ref{l:lower} below. 
\end{rem}
\begin{proof}
  We argue as in the proof of Lemma~\ref{l:interpol} but we use H\"older's inequality with weights.
  We consider $p_\eps = p -\eps <p$ and write $\frac1{q_1} = \frac{\alpha_\eps}{p_\eps} + \frac{1-\alpha_\eps}{p p_0}$.
  Since $p_\eps <p$, we have $\alpha_\eps > \alpha$. 
  We then write $0 = (1-\alpha_\eps) (k_0 /p) + \alpha_\eps k_1$,  and get from Lemma~\ref{l:holder-weight}
\[
  \|f\|_{L^{q_1}}^{p+1}   \le \|f\|_{L^{p_\eps}_{k_1}}^{\alpha_\eps(p+1)} \|f\|_{L^{pp_0}_{k_0/p}}^{(1-\alpha_\eps)(p+1)}.
\]
Using once again H\"older's inequality for product but with $P_\eps$ such that
\[P_\eps=(1-\alpha_\eps)^{-1} (p+1)^{-1}p \to P > 1 \qquad \text{ as } \eps \to 0,\]
we get
\[ \|f\|_{L^{q_1}}^{p+1}    \le C_{\eps} \|f\|_{L^{p_\eps}_{k_1}}^{\alpha_\eps (p+1)Q_\eps} +  \eps \|f^p\|_{L^{p_0}_{k_0}}.\]
We have  $\alpha_\eps (p+1) Q_\eps  = \beta_\eps p$ with
\[ \beta_\eps = \frac{\alpha_\eps(p+1)}{\alpha_\eps (p+1)-1} = 1 + \frac1{\alpha_\eps (p+1)-1} \to  1 + \frac1{\alpha (p+1)-1} = \beta \text{ as } \eps \to 0. \]
 We now use H\"older's inequality with weights, see Lemma~\ref{l:holder-weight}, in order to get
the desired result. 
\end{proof}

We will need the following technical estimate in the proof of Theorem~\ref{t:prodi-serrin}.
\begin{lem}[A second interpolation estimate with weights]\label{l:interpol-again}
  Let $p \in \left(\frac{d}{d+\gamma+2s},\frac{d}{d+\gamma} \right)$
  and  $\frac1{r_p} = \frac1p \left( 2 + \frac{\gamma}d - \frac1{p} \right)$.
  Then $r_p \in [p/2,pp_0]$ and
\[
 \left\| F\right\|_{L^{r_p}}^p  \le C_\eps^0 \|F\|_{L^1_{k_p}} \int_{\R^d} F^{p-1} \langle v \rangle^{k_p} \dv   + \eps \left\| F^p\right\|_{L^{p_0}_{k_0}} 
\]
for some $k_p >0$. 
\end{lem}
\begin{proof}
We first apply H\"older's inequality with weights, see Lemma~\ref{l:holder-weight}, with $r_p \in [p/2,p p_0]$ and get for some $k_p>0$ such that,
\[
 \left\| F\right\|_{L^{r_p}}^p  \le  C_\eps^0 \left\| F\right\|_{L^{\frac{p}2}_{k_p}}^p + \eps \left\| F\right\|_{L^{pp_0}_{k_0/p}}^p.
\]
The last line also holds true if $p<2$ with $\|F \|_{L^{p/2}} := \| \sqrt{F}\|_{L^p}^2$.
Since we also have,
\[ \|F \|_{L^{\frac{p}2}_{k_p}} \le \|F\|_{L^1_{k_p}}^{\frac1p} \|F\|_{L^{p-1}_{k_p}}^{1-\frac1p},\]
we get the announced inequality.
\end{proof}

\section{Coercivity estimates}
\label{s:coercivity}

In this section, we introduce a family of increasing convex functions  that are sublinear and approximate $r \mapsto r^p$ for any given $p>1$.
Then we derive two lower bounds for the associated dissipation. 

\subsection{Approximate $L^p$ norms and  associated Bregman distances}
\label{ss:phik}

\paragraph{The approximate $L^p$ norm.}

We consider $\phik (r)$ convex such that
\[ \phik (0) = \phik'(0)=0 \text{ and } \phik'' (r) = p (p-1) r^{p-2} \un_{\{r \le \kappa\}}.\]
In order words,
\[ \phik (r) = (r \wedge \kappa)^p + p \kappa^{p-1} (r - \kappa)_+.\]

\paragraph{The function $\Phi_\kappa$.}
When $\varphi = \varphi_k$, the function $\Phi$ appearing in Definition~\ref{d:suitable} equals
\begin{equation} \label{e:Phi_k}
  \Phi_\kappa(f) = (p-1) (f \wedge \kappa)^p.
\end{equation}

\paragraph{The Bregman distance.}
In this case, we have
\[
  d_{\phik} (r,s) =
  \begin{cases}
    s^p -r^p - pr^{p-1} (r-s) & \text{ if } r,s \le \kappa, \\
    p (\kappa^{p-1}- r^{p-1}) s + (p-1) (r^p - \kappa^p) &\text{ if } r \le \kappa \le  s .
  \end{cases}
\]
We compute that,
\begin{equation}
  \label{e:breg-1}
  d_{\phik} (r,r/2) =    C_{1,p} r^p  \quad \text{ if } r \le \kappa, 
\end{equation}
with $C_{1,p} = 2^{-p} - 1 + p/2 >0$ for $p >1$. In the case where $2 r < \kappa < s$, we write
\begin{align*}
  d_{\phik} (r,s) &= p (\kappa^{p-1}- r^{p-1}) s + (p-1) (r^p - \kappa^p) \\
                  & = p (\kappa^{p-1}- r^{p-1}) (s-\kappa) + \kappa^p - p \kappa r^{p-1} + (p-1) r^p \\
  & =  p (\kappa^{p-1}- r^{p-1}) (s-\kappa) + \kappa^p \Gamma (r/\kappa) \\
\intertext{with $\Gamma (\rho ) = (p-1) \rho^p - p \rho^{p-1} + 1$. Remark that the function $\Gamma$ is non-increasing in $(0,1)$,}
                  & \ge p (1-2^{1-p}) \kappa^{p-1} (s-\kappa) + \Gamma (1/2) \kappa^p.
\end{align*}
We thus have,
\begin{equation}
  \label{e:breg-2}
 d_{\phik} (r,s) \ge C_{2,p} (\phik (s) + \kappa^p)  \quad \text{ if } s > \kappa \text{ and } r < \kappa/2.
\end{equation}
where $C_{2,p} = \min (1-2^{1-p}, (p-1)2^{-p} - p 2^{1-p}+1)> 0$. 

\subsection{Approximate $L^p$ dissipation estimate}

\begin{lem}[First lower bound for the dissipation] \label{l:lower}
\[
  \iint_{\R^d \times \R^d} d_{\phik} (f,f') K (v,v') \dv \dv'
 \ge  c_{1,D} \| (f \wedge \kappa)^p \|_{L^{p_0}_{k_0}} - C_{1,D} 
    \int_{\R^d} (f\wedge \kappa)^p (v) \langle v \rangle^{\gamma} \dv .
\]
where $p_0>1$ is such that $\frac1{p_0} = 1 - \frac{2s}d$ and  $k_0 = \gamma + 2s - \frac{2s}d<0$. 
  The constants $c_{1,D}$ and $C_{1,D}$ only depend on hydrodynanical bounds from \eqref{e:hydro}, dimension $d$, $\gamma,s$ and $p$.
\end{lem}

\begin{proof}
  We first reduce the study to velocities $v$ and $v'$ such that $f(v') < (f\wedge \kappa)(v)/2$.
\begin{align}
\nonumber  \iint_{\R^d \times \R^d} & d_{\phik} (f,f') K (v,v')  \dv \dv' \\
\nonumber  &\ge   \iint_{\{f'< (f \wedge \kappa)/2\}} d_{\phik} (f,f') K (v,v') \dv \dv' \\
  \intertext{we use monotonicity properties of the Bregman distance, see Lemma~\ref{l:breg},}
\nonumber                           &\ge  \int_{\R^d}  d_{\phik} ((f \wedge \kappa),(f \wedge \kappa)/2) \left\{ \int_{\{f'< (f\wedge \kappa)/2\}} K (v,v') \dv' \right\} \dv \\
\intertext{we use  the fact that $d_{\phik}((f\wedge \kappa)(v),(f\wedge \kappa)(v)/2) = C_{1,p} (f\wedge \kappa)(v)^p$, see \eqref{e:breg-1},}
\nonumber                           &\ge  C_{1,p} \int_{\R^d}  (f \wedge \kappa)^p \left\{ \int_{\{f'< (f\wedge \kappa)/2\}} K (v,v') \dv' \right\} \dv \\  
  \intertext{we use the non-degeneracy cone to estimate $K(v,v')$ from below, see Lemma~\ref{l:cone},}
\nonumber  &\ge \lambda_0 C_{1,p} \int_{\R^d}  (f \wedge \kappa)^p  \langle v \rangle^{\gamma+2s+1} \left\{ \int_{\{f'< (f\wedge \kappa)/2\} \cap C(v)}  |v'-v|^{-d-2s} \dv' \right\} \dv \\
  \intertext{we further reduce the set of integration for velocities $v'$ by considering a ball $B_R (v)$,}
\label{e:ast}  &\ge \lambda_0 C_{1,p} \int_{\R^d}  (f \wedge \kappa)^p(v) \langle v \rangle^{\gamma+2s+1} R^{-d-2s} |\{f'< (f\wedge \kappa)/2\} \cap C(v) \cap B_R (v)\}| \dv.
\end{align}
We continue reducing the domain of integration in $v$ by considering
\[ G = \{ v\in \R^d : \langle v \rangle \ge R \}\]
with some $R$ to be chosen. 

\paragraph{Choice of $R$.} We  choose $R$ in order to ensure (more or less) that $f' < (f \wedge \kappa)/2$ in more than the half of $C(v) \cap B_R(v)$. More precisely, we know from Lemma~\ref{l:cone} that
$|C(v) \cap B_R (v) | \ge \lambda_0 R^d \langle v \rangle^{-1}$ and we aim at finding $R$ such that
$  | \{ f' < (f\wedge \kappa)(v)/2 \} \cap B_R (v) \cap C(v) |
\ge \frac{\lambda_0}2 \frac{R^d}{\langle v \rangle} .$

Let $N$ denote $\| (f \wedge \kappa)^p \|_{L^{p_0}_{k_0}}$.   We pick $R>0$ such that
\begin{equation}
  \label{e:R}
  R^d = C_R \frac{N^{p_0}}{ (f \wedge \kappa)(v)^{p p_0} \langle v \rangle^{k_0 p_0-1}}
  \quad \text{ with } \quad N = \|(f \wedge \kappa)^p\|_{L^{p_0}_{k_0}},
\end{equation}
For $v \in G$ and $v' \in B_R (v)$, we have $\langle v' \rangle \le 2 \langle v \rangle$.
With such a piece of information at hand, we can use  Chebyshev's inequality to get,
\begin{align*}
 | \{ (f'\wedge \kappa) \ge (f\wedge \kappa)(v)/2 \} \cap B_R (v) \cap C(v) |  & \le \frac{2^{p p_0}}{(f\wedge \kappa)^{p p_0}(v)} \int_{\{\langle v' \rangle \le 2 \langle v \rangle\}} (f \wedge \kappa)^{p p_0}(v') \dv' \\
                                                                               & \le \frac{2^{p p_0+ k_0 p_0}}{ (f \wedge \kappa)(v)^{p p_0} \langle v \rangle^{k_0 p_0}} \int_{\R^d}  (f \wedge \kappa)^{p p_0}(v') \langle v' \rangle^{k_0p_0} \dv' \\
                                                                               & \le \frac{2^{pp_0 + k_0p_0}}{(f \wedge \kappa)(v)^{p p_0} \langle v \rangle^{k_0p_0}} N^{p_0} \\
  & = \left(\frac{2^{pp_0+k_0p_0}}{C_R}\right) \frac{R^d}{\langle v \rangle}.
\end{align*}
We used $k_0p_0 <0$ to get the second line.

\paragraph{Estimate of the measure of the sublevel set.}
We now choose $C_R = 2^{pp_0 + k_0p_0+1}/\lambda_0$ and get,
\[  | \{ (f'\wedge \kappa) \ge (f\wedge \kappa)(v)/2 \} \cap B_R (v) \cap C(v) | \le \frac{\lambda_0}2 \frac{R^d}{\langle v \rangle} .\]  
We continue by remarking that, since $C(v)$ is a cone, we have $|C(v) \cap B_R (v) | = |C(v) \cap B_1 (v) | R^d$
and we have from Lemma~\ref{l:cone} that $|C(v) \cap B_1 (v) | \ge \lambda_0 \langle v \rangle^{-1}$.
In particular, we can write,
\begin{equation}
  \label{e:ast-ast}
\begin{aligned}
  |\{f'< (f\wedge \kappa)/2\} \cap C(v) \cap B_R (v)\}|
  & = |C(v) \cap B_R (v)| - |\{f'\ge (f\wedge \kappa)/2\} \cap C(v) \cap B_R (v)\}| \\
  & \ge \frac{\lambda_0}2 \frac{R^d}{\langle v \rangle}  .
\end{aligned}
\end{equation}

\paragraph{Conclusion.}
We continue with the coercivity estimate. We use \eqref{e:ast}, \eqref{e:R} and \eqref{e:ast-ast} to get,
\[  
  \iint_{\R^d \times \R^d} d_{\phik} (f,f') K (v,v') \dv \dv' \ge 
  C_{1,p} \frac{\lambda_0^2}{2} C_R^{-\frac{2s}d} N^{-\frac{2s}d p_0} \int_{G}  (f\wedge \kappa)^{p(1+\frac{2s}d  p_0)} \langle v \rangle^{\gamma+2s + \frac{2s}d(k_0 p_0-1)}    \dv.
\]
We remark that $-\frac{2s}d p_0 = 1 -p_0$,   $p(1+\frac{2s}d  p_0) = p p_0$ and $\gamma+2s + \frac{2s}d(k_0 p_0-1) = k_0 p_0$.
We thus get
\[  \iint_{\R^d \times \R^d} d_{\phik} (f,f') K (v,v') \dv \dv' \ge c_{1,D} \| (f \wedge \kappa)^p \|_{L^{p_0}_{k_0}}^{1-p_0} \int_{G} (f\wedge \kappa)^{p p_0} (v) \langle v \rangle^{k_0 p_0} \dv \]
with $c_{1,D} = C_{1,p} \frac{\lambda_0^2}{2} C_R^{-\frac{2s}d}$  where
  \(G = \{ v \in \R^d : (f \wedge \kappa)^{p p_0} (v) \langle v \rangle^{k_0 p_0} \ge C_G \| (f \wedge \kappa)^p \|_{L^{p_0}_{k_0}}^{p_0} \langle v \rangle^{1-d}\}.\)
  If we consider $B = \R^d \setminus G$ then we have,
\begin{align*}   
\int_{B} (f\wedge \kappa)^{p p_0} (v) \langle v \rangle^{k_0 p_0} \dv
  &\le C_G^{1-\frac1{p_0}} \| (f \wedge \kappa)^p \|_{L^{p_0}_{k_0}}^{p_0-1} \int_{B} (f\wedge \kappa)^{p} (v) \langle v \rangle^{k_0 +(1-d) (1-\frac1{p_0})} \dv \\
  \intertext{since $\frac1{p_0} = 1- \frac{2s}d$ and $k_0 = \gamma + 2s - \frac{2s}d$, we get}
&\le C_G^{1-\frac1{p_0}} \| (f \wedge \kappa)^p \|_{L^{p_0}_{k_0}}^{p_0-1} \int_{B} (f\wedge \kappa)^{p} (v) \langle v \rangle^{\gamma} \dv.
\end{align*}
From this upper bound, we get,
\[
    \iint_{\R^d \times \R^d} d_{\phik} (f,f') K (v,v') \dv \dv'
   \ge  c_{1,D} \| (f \wedge \kappa)^p \|_{L^{p_0}_{k_0}} - c_{1,D} C_G^{1-\frac1{p_0}}
    \int_{\R^d} (f\wedge \kappa)^p (v) \langle v \rangle^{\gamma} \dv .
\]
We get the desired lower bound with $C_{1,D} = c_{1,D} C_G^{1-\frac1{p_0}}$. 
\end{proof}

\begin{lem}[Second lower bound for the dissipation] \label{l:lower-bis}
Let $p \in (1,\infty)$. There exists $\kappa_0 \ge 1$ only depending on $d,\gamma,s,p$ and the bounds in \eqref{e:hydro} such that for all $\kappa \ge \kappa_0$,
  \[ \iint_{\R^d \times \R^d} d_{\phik} (f,f') K (v,v') \dv \dv' \ge c_{2,D} \kappa^{\frac{2sp }d} \| (f \wedge \kappa)\|_{L^p}^{-\frac{2sp}d} \int_{\R^d} \kappa^{p-1}(f-\kappa)_+ (v) \langle v \rangle^{k_0} \dv \]
  where we recall that $k_0 = \gamma +2s-\frac{2s}d<0$.
  The constant $c_{2,D}$ also only depends on hydrodynanical bounds, dimension, $s$, $\gamma$ and $p$.
\end{lem}

\begin{proof}
  We first reduce the study to velocities $v$ and $v'$ such that $f(v') > \kappa > 2 f(v)$. We then can use the lower bound on the Bregman distance given by~\eqref{e:breg-2} in order
  to write,
\begin{align*}
  \iint_{\R^d \times \R^d}  d_{\phik} (f,f') K (v,v')  \dv \dv' &
  \ge   \iint_{\{f'> \kappa >  2 f \}} d_{\phik} (f,f') K (v,v') \dv \dv' \\
  &\ge  C_{2,p} \int_{\{f' > \kappa\}}  \phik (f') \left\{ \int_{\{f< \kappa/2 \}} K (v,v') \dv \right\} \dv'.
\end{align*}

\paragraph{Use of the non-degeneracy set.}
Given $v' \in \R^d$, we consider the non-degeneracy set
\[ S (v') = \{ v \in \R^d : K(v,v') \ge \lambda_0' \langle v' \rangle^{\gamma +2s +1} |v'-v|^{-d-2s} \}.\]
We use this set in order to continue the previous computation,
\begin{align*}
  \iint_{\R^d \times \R^d}  & d_{\phik} (f,f') K (v,v')  \dv \dv' \\
& \ge C_{2,p} \mu_0 \int_{\{f' > \kappa\}}  \phik (f') \langle v' \rangle^{1+\gamma+2s} \left\{ \int_{\{f< \kappa/2 \} \cap S (v') \cap B_R (v')} |v'-v|^{-d-2s} \dv \right\} \dv' \\
& \ge C_{2,p} \mu_0 \int_{\{f' > \kappa\}}  \phik (f') \langle v' \rangle^{1+\gamma+2s} |\{f< \kappa/2 \} \cap S (v') \cap B_R (v')| R^{-d-2s}  \dv'. 
\end{align*}

\paragraph{Choice of $R$.} We choose $R=R(v')$ so that
\[ \frac{R^d}{\langle v' \rangle} = C_0 \frac{\|(f\wedge \kappa)\|_{L^p}^p}{\kappa^p}\]
for some $C_0>0$ to be chosen later. 
We need to use the non-degeneracy of $K_f$ as expressed in Lemma~\ref{l:sphere}. We thus need such an $R$
to satisfy
\[R \le \eps_0 \langle v' \rangle   \]
or equivalently,
\[ \langle v' \rangle^{d-1} \ge \eps_0^{-d} C_0 \frac{N^p}{\kappa^p} \quad \text{ with } \quad N = \|(f\wedge \kappa)\|_{L^p}.\]
We further remark that
\[ N^p \le \kappa^{p-1} M_0 \] and we conclude that the previous inequality holds true as soon as,
\[ \langle v' \rangle^{d-1} \ge \eps_0^{-d} C_0 \frac{M_0}{\kappa}. \]
We thus choose $C_0$ and $\kappa_0$ such that
\[  C_0 \frac{M_0}{\eps_0^d \kappa_0} \le 1.\]

\paragraph{Estimate of the sub-level set.}
Keeping in mind that we chose $R$ in such a way that $3R \le \langle v' \rangle$, we use Lemma~\ref{l:sphere} and
Chebyshev's inequality in order to get for $v'$ such that $R \ge 2 R_0$,
\begin{align*}
  |\{f< \kappa/2 \} \cap S (v') \cap B_R (v')| &\ge \mu_0 \frac{R^d}{\langle v'\rangle} - |\{f< \kappa/2 \} \cap S (v') \cap B_R (v')|\\
                                               & \ge \mu_0 \frac{R^d}{\langle v'\rangle} - \frac{2^p}{\kappa^p} N^p \\
  & \ge \mu_0 \frac{R^d}{\langle v'\rangle} - \frac{2^p}{C_0}  \frac{R^d}{\langle v'\rangle}.
\end{align*}
We then choose $C_0 = 2^{p+1}/\mu_0$ (and consequently $\kappa_0 = 2^{p+1} M_0 /(\eps_0^d \mu_0)$) and get
\[  |\{f< \kappa/2 \} \cap S (v') \cap B_R (v')| \ge \frac12 \mu_0 \frac{R^d}{\langle v'\rangle}.\]

\paragraph{Conclusion.}
We can now turn back to the dissipation estimate. Recalling that $R^d = C N^p \kappa^{-p} \langle v' \rangle$, we get,
\begin{align*}
  \iint_{\R^d \times \R^d}  & d_{\phik} (f,f') K (v,v')  \dv \dv' \\
                            & \ge C_{2,p} \frac{\mu_0^2}2  \int_{\{f' > \kappa\}}  \phik (f') \langle v' \rangle^{\gamma+2s}  R^{-2s}  \dv' \\
\intertext{we remark that $\phik (f') \un_{\{f' > \kappa\}} = \kappa^p + p \kappa^{p-1} (f'-\kappa)_+$,}
                            & \ge C_{2,p} \frac{\mu_0^2}2 C_0^{-\frac{2s}d} p \kappa^{\frac{2s}d p+p-1} \| (f\wedge \kappa)\|_{L^p}^{-\frac{2s}d p} \int_{\R^d}   (f'-\kappa)_+ \langle v' \rangle^{\gamma+2s - \frac{2s}d}    \dv'.
\end{align*}
We get the result with $c_{2,D} = C_{2,p} \mu_0^2 C_0^{-\frac{2s}d} p /2$.
\end{proof}

\section{Propagation of approximate $L^p$ norms}
\label{s:propagation}

\subsection{Source terms vs. dissipation}

\begin{lem}[Leading source term vs. dissipation]\label{l:leading}
Assume that $f \colon \R^d \to \R$ has moments of any order. Then,
  \[ \int_{\R^d} (f \wedge \kappa)^p ( (f\wedge \kappa) \ast_v |\cdot|^\gamma) \dv
    \le  \eps \| (f\wedge \kappa)^p\|_{L^{p_0}_{k_0}}   +  \bar{C}_\eps \|(f \wedge \kappa) \|_{L^p}^{p \beta_\eps} + \bar{C}_\eps (\Cprop{k_\eps})^{r_\eps}\]
  where $\bar{C}_\eps$  depends on $\eps$, $c_c$, $p$, $d$, $\gamma$ and  moments of $f$.
\end{lem}
\begin{proof}
  We apply Lemma~\ref{l:convol} to  the convolution product,
\begin{align*}
  \int_{\R^d} (f \wedge \kappa)^p ( (f\wedge \kappa) \ast_v |\cdot|^\gamma) \dv
  & \le  C_{d,\gamma} \| (f \wedge \kappa)^p \|_{L^{q_1/p}} \|(f \wedge \kappa) \|_{L^{q_1}} \\
  & = C_{d,\gamma}  \|(f \wedge \kappa) \|_{L^{q_1}}^{p+1} \\
  \intertext{we  apply
  Lemma~\ref{l:interpol-weight} to $(f \wedge \kappa)$  for $\eps>0$ arbitrarily small and
  $\beta$ given by \eqref{e:beta},}
  & \le C_{d,\gamma} \eps \| (f\wedge \kappa)^p\|_{L^{p_0}_{k_0}}
    + C_{d,\gamma} C_\eps \|(f \wedge \kappa) \|_{L^p}^{p \beta_\eps} + C_{d,\gamma} C_\eps \|(f \wedge \kappa)\|^{r_\eps}_{L^1_{k_\eps}} \\
  \intertext{we  use  (the propagation of) moments, recall Theorem~\ref{t:moments},}
  & \le C_{d,\gamma} \eps \| (f\wedge \kappa)^p\|_{L^{p_0}_{k_0}}
    + C_{d,\gamma} C_\eps \|(f \wedge \kappa) \|_{L^p}^{p \beta_\eps} + C_{d,\gamma} C_\eps (\Cprop{k_\eps})^{r_\eps}. \qquad \qedhere
\end{align*}
\end{proof}
\begin{lem}[Error term vs. dissipation - I]\label{l:error}
Assume that $f \colon \R^d \to \R$ has moments of any order. Then,
  \begin{align*}
    \int_{\R^d} (f \wedge \kappa)^p  \bigg( (f-\kappa)_+ \ast_v |\cdot|^\gamma \bigg) \dv
    &\le \eps  \kappa^{\frac{2sp}d} \left\|(f \wedge \kappa) \right\|^{- \frac{2s}d p}_{L^p} \int_{\R^d} \kappa^{p-1} (f-\kappa)_+(v) \langle v \rangle^{k_0} \dv \\
    &+  \tilde{C}_\eps \kappa^{-\iota} \left(\int_{\R^d} \phik (f) \dv \right)^{\tilde{\beta}} 
\end{align*}
where  $\tilde{C}_\eps$ depends on $\eps$, $c_c$, $p$, $d$, $\gamma$ and  moments of $f$.
The positive constants  $\tilde \beta$ and $\iota$ are given by,
\[
  \begin{cases}
         \tilde \beta = \frac{\alpha_p}{1-\theta_p} \\
     \iota =(p-1)\frac{1 - \frac{\theta_p}{1-\eps}}{1-\theta_p}
  \end{cases}
   \quad \text{ with } \quad \begin{cases} 
    \theta_p = \frac{-\frac{\gamma p}{d}  }{(\frac{2s}d+1)p -1} \\
     \alpha_p = 1 - \frac{\theta_p}{1-\eps} + \frac{2s}d \theta_p + (1+ \frac{\gamma}d).
    \end{cases}\]

\end{lem}
\begin{proof}
We have 
\begin{align*}
  \int_{\R^d} (f \wedge \kappa)^p  \bigg( (f-\kappa)_+ \ast_v |\cdot|^\gamma \bigg) \dv
  &= \int_{\R^d} \bigg((f \wedge \kappa)^p  \ast_v |\cdot|^\gamma \bigg) (f-\kappa)_+ \dv \\
  & \le  \left\|(f \wedge \kappa)^p  \ast_v |\cdot|^\gamma \right\|_{L^\infty(\R^d)} \int_{\R^d} (f-\kappa)_+ \dv.
\end{align*}
On the one hand, using the weak Young inequality and interpolation in Lebesgue spaces, we have
\begin{align}
\nonumber
  \left\|(f \wedge \kappa)^p  \ast_v |\cdot|^\gamma \right\|_{L^\infty(\R^d)} 
& \le C_{d,\gamma} \left\|(f \wedge \kappa)^p \right\|_{L^{\frac{d}{d+\gamma}}}   \\
\nonumber  & \le C_{d,\gamma} \left\|(f \wedge \kappa)^p \right\|_{L^\infty}^{-\frac{\gamma}d} \left\|(f \wedge \kappa)^p \right\|^{1+\frac{\gamma}d}_{L^1}  \\
\label{e:power-k}  & \le C_{d,\gamma} \kappa^{-\frac{\gamma p}d} \left\|(f \wedge \kappa) \right\|^{p(1+\frac{\gamma}d)}_{L^p}  .
\end{align}
We recall that $C_{d,\gamma}$ is given in Remark~\ref{r:weak-young}.
On the other hand, for any $\eps \in (0,1)$ and $k_0 >0$ appearing in the second entropy dissipation estimate, see Lemma~\ref{l:lower-bis}, we have,
\begin{align*}
  \int_{\R^d} (f-\kappa)_+ \dv & = \int_{\R^d} (f-\kappa)_+^{1-\eps}(v) \langle v \rangle^{(1-\eps) k_0} (f-\kappa)_+^{\eps}(v)  \langle v \rangle^{(\eps-1) k_0}  \dv \\
  & \le \left(\int_{\R^d} (f-\kappa)_+(v) \langle v \rangle^{k_0} \dv \right)^{1-\eps} \left(\int_{\R^d} f(v)\langle v \rangle^{\frac{\eps-1}{\eps}k_0} \dv \right)^\eps.
\end{align*}
We can now use (the propagation of) moments, recall Theorem~\ref{t:moments}, to get for all $\eps \in (0,1)$, 
\[   \int_{\R^d} (f-\kappa)_+ \dv  \le(\Cprop{(1-\eps^{-1})k_0})^\eps \left(\int_{\R^d} (f-\kappa)_+(v) \langle v \rangle^{k_0} \dv \right)^{1-\eps} .\]

For $p > \frac{d}{d+\gamma+2s} >1$, we consider
\[ \theta_p = \frac{-\frac{\gamma p}{d}  }{(\frac{2s}d+1)p -1} \in (0,1).\]
Our goal is to make appear the second lower bound for the entropy dissipation estimate obtained  in Lemma~\ref{l:lower-bis}. It will appear with a power $\theta_p$: it is enough to compare
the exponents for $\kappa$ in \eqref{e:power-k} and in the lower bound in Lemma~\ref{l:lower-bis}  to see it.

Coming back to the estimate of the error term, we now write precisely,
\begin{align*}
 c_c (p-1) \int_{\R^d} & (f \wedge \kappa)^p  \bigg( (f-\kappa)_+ \ast_v |\cdot|^\gamma \bigg) \dv \\
   \le C_4 &  \kappa^{-\frac{\gamma p}d} \left\|(f \wedge \kappa) \right\|^{p(1+\frac{\gamma}d)}_{L^p}
  \left(\int_{\R^d} (f-\kappa)_+(v) \langle v \rangle^{k_0} \dv \right)^{\theta_p}   \left(\int_{\R^d} (f-\kappa)_+(v)  \dv \right)^{1-\frac{\theta_p}{1-\eps}} \\
\intertext{with $C_4 = c_c (p-1) (\Cprop{(1-\eps^{-1})k_0})^{\frac{\eps \theta_p}{1-\eps}} C_{d,\gamma}$,}
  \le C_4 & \left( \kappa^{(\frac{2s}d+1)p-1} \left\|(f \wedge \kappa) \right\|^{- \frac{2s}d p}_{L^p} \int_{\R^d} (f-\kappa)_+(v) \langle v \rangle^{k_0} \dv \right)^{\theta_p} \\
          & \times  \left\|(f \wedge \kappa) \right\|^{p [ (\frac{2s}d ) \theta_p + (1+\frac{\gamma}d) ]}_{L^p} \left(\int_{\R^d} (f-\kappa)_+(v)  \dv \right)^{1 - \frac{\theta_p}{1-\eps}}. 
\end{align*}
We now use the fact that $(f \wedge \kappa)^p \le \phik (f)$ and $(f-\kappa)_+ \le \kappa^{1-p} \phik (f)$ (we recall that $p \ge 1$) and get for
\[ \alpha_p = \left[\frac{2s}d \theta_p + (1+ \frac{\gamma}d) \right] + 1 - \frac{\theta_p}{1-\eps},\]
the following inequality,
\begin{align*}
c_c   & (p-1) \int_{\R^d}   (f \wedge \kappa)^p  \bigg( (f-\kappa)_+ \ast_v |\cdot|^\gamma \bigg) \dv \\
  \le & C_4 \kappa^{(1-p)(1 - \frac{\theta_p}{1-\eps})} \left( \kappa^{(\frac{2s}d+1)p-1} \left\|(f \wedge \kappa) \right\|^{- \frac{2s}d p}_{L^p} \int_{\R^d} (f-\kappa)_+(v) \langle v \rangle^{k_0} \dv \right)^{\theta_p}
            \left(\int_{\R^d} \phik (f) (v) \dv \right)^{\alpha_p}  \\
  \le & \eps  \kappa^{(\frac{2s}d+1)p-1} \left\|(f \wedge \kappa) \right\|^{- \frac{2s}d p}_{L^p} \int_{\R^d} (f-\kappa)_+(v) \langle v \rangle^{k_0} \dv
        +  C_5 \kappa^{(1-p)\frac{1 - \frac{\theta_p}{1-\eps}}{1-\theta_p}} \left(\int_{\R^d} \phik (f) (v) \dv \right)^{\frac{\alpha_p}{1-\theta_p}} 
\end{align*}
with $C_5= C_4^{\frac1{1-\theta_p}} \eps^{-\frac{\theta_p}{1-\theta_p}}$.
\end{proof}

We need a variation of the previous lemma in the proof of Theorem~\ref{t:prodi-serrin}.
\begin{lem}[Error term vs. dissipation - II]\label{l:error-II}
Assume that $f \colon \R^d \to \R$ is such that $\|f\|_{L^p} \le C_0$ for $p \in \left( \frac{d}{d+\gamma+2s},\frac{d}{d+\gamma}\right)$ and $F$ has moments of any order. Then,
  \begin{align*}
    \int_{\R^d} (f \wedge a) (F \wedge \kappa)^{p-1}  \bigg( (F-\kappa)_+ \ast_v |\cdot|^\gamma \bigg) \dv
    &\le \eps  \kappa^{\frac{2sp}d} \left\|(F \wedge \kappa) \right\|^{- \frac{2s}d p}_{L^p} \int_{\R^d} \kappa^{p-1} (F-\kappa)_+(v) \langle v \rangle^{k_0} \dv \\
    &+  \tilde{C}_\eps \kappa^{-\iota} \left(\int_{\R^d} \phik (F) \dv \right)^{\bar{\beta}} 
\end{align*}
where  $\tilde{C}_\eps$ depends on $\eps$, $c_c$, $p$, $d$, $\gamma$, $C_0$ and  moments of $f$.
The positive constants  $\bar \beta$ and $\bar \iota$ are given by,
\[
  \begin{cases}
         \bar \beta = \frac{\bar \alpha_p}{1-\theta_p} \\
     \iota =(p-1)\frac{1 - \frac{\theta_p}{1-\eps}}{1-\theta_p}
  \end{cases}
   \quad \text{ with } \quad \begin{cases} 
    \theta_p = \frac{-\frac{\gamma p}{d}  }{(\frac{2s}d+1)p -1} \\
     \bar \alpha_p = 1 - \frac{\theta_p}{1-\eps} + \frac{2s}d \theta_p + \left( 1 +\frac{\gamma}d - \frac1p\right).
    \end{cases}\]
\end{lem}
\begin{proof}
We have 
\begin{align*}
  \int_{\R^d} (f \wedge a) (F \wedge \kappa)^{p-1}  \bigg( (F-\kappa)_+ \ast_v |\cdot|^\gamma \bigg) \dv
  &= \int_{\R^d} \bigg((f \wedge a) (F \wedge \kappa)^{p-1}  \ast_v |\cdot|^\gamma \bigg) (f-\kappa)_+ \dv \\
  & \le  \left\|(f \wedge a) (F \wedge \kappa)^{p-1}  \ast_v |\cdot|^\gamma \right\|_{L^\infty(\R^d)} \int_{\R^d} (f-\kappa)_+ \dv.
\end{align*}
On the one hand, using the weak Young inequality and interpolation in Lebesgue spaces, we have
\begin{align*}
\nonumber
  \left\| (f \wedge a) (F \wedge \kappa)^{p-1} \ast_v |\cdot|^\gamma \right\|_{L^\infty(\R^d)} 
& \le C_{d,\gamma} \left\|(f \wedge a) (F \wedge \kappa)^{p-1} \right\|_{L^{\frac{d}{d+\gamma}}}   \\
& \le C_{d,\gamma} \|(f \wedge a)\|_{L^p} \left\| (F \wedge \kappa)^{p-1} \right\|_{L^{\bar p}}   \\
  \intertext{with $\frac1{\bar p} = 1 + \frac{\gamma}d - \frac1p$, we then use the assumption on $f$ and interpolation,}
  \nonumber  & \le C_{d,\gamma} C_0 \left\|(F \wedge \kappa)^{p-1} \right\|_{L^\infty}^{1-\theta} \left\|(F \wedge \kappa)^{p-1} \right\|^{\theta}_{L^{\frac{p}{p-1}}}  \\
  \intertext{with $\theta = \frac{p}{\bar p (p-1)}$, in particular $(p-1)(1-\theta) = - \frac{\gamma}d p$,}
  & \le C_{d,\gamma} C_0 \kappa^{-\frac{\gamma p}d} \left\|(F \wedge \kappa) \right\|^{\frac{p}{\bar p}}_{L^p}  .
\end{align*}
Then the proof is exactly the same as the one of the previous lemma. In particular, since the exponent in $\kappa$ in
the last inequality coincides with the one appearing in the previous proof, the exponent $\theta_p$ is unchanged (and in
turn $\iota$) is unchanged. 
\end{proof}

\subsection{Dissipation of approximate $L^p$ norms}

We can now combine the two previous lemmas with the coercivity estimates  established in the previous section in order to get
the following key estimate.
\begin{lem}[Source and error terms vs. dissipation]\label{e:key-estim}
Let $f \colon \R^d \to \R$ with moments of any order. Then,
\begin{multline*}
  c_c \int_{\R^d} \Phik (f) (f \ast_v |\cdot|^\gamma) \dv - \int_{\R^d} d_{\phik} (f,f') K (v,v') \dv \dv' \\
\le  C_3 + C_3  \int_{\R^d} \phik (f) \dv  + C_3 \left( \int_{\R^d} \phik (f) \dv \right)^{\beta_\eps}
  + C_3 \kappa^{-\iota} \left(\int_{\R^d} \phik (f) \dv \right)^{\tilde \beta}
\end{multline*}
with $\beta = 1 + \frac{2s}{p(d+\gamma+2s)-d} >1$ and $\tilde \beta>0$ and $\iota>0$ given by,
\[
  \begin{cases}
         \tilde \beta = \frac{\alpha_p}{1-\theta_p} \\
     \iota =(p-1)\frac{1 - \frac{\theta_p}{1-\eps}}{1-\theta_p}
  \end{cases}
   \quad \text{ with } \quad \begin{cases} 
    \theta_p = \frac{-\frac{\gamma p}{d}  }{(\frac{2s}d+1)p -1} \\
     \alpha_p = 1 - \frac{\theta_p}{1-\eps} + \frac{2s}d \theta_p + (1+ \frac{\gamma}d).
    \end{cases}\]
  \end{lem}
\begin{proof}
We split $f = (f-\kappa)_+ + f \wedge \kappa$ and write,
\begin{align*}
  c_c \int_{\R^d} \Phik (f) (f \ast_v |\cdot|^\gamma) \dv & \le c_c  \int_{\R^d} \Phik (f) ( (f\wedge \kappa) \ast_v |\cdot|^\gamma) \dv + c_c \int_{\R^d} \Phik (f) ( (f-\kappa)_+ \ast_v |\cdot|^\gamma) \dv \\
\intertext{we now use the fact that $\Phik(f)=(p-1)(f\wedge \kappa)^p$ and apply Lemmas~\ref{l:leading} and \ref{l:error},}
                                                          & \le  c_c (p-1) \left(\eps \| (f\wedge \kappa)^p\|_{L^{p_0}_{k_0}}   +  C_\eps \|(f \wedge \kappa) \|_{L^p}^{p \beta_\eps} + C_\eps (\Cprop{k_\eps})^{r_\eps} \right)\\
  & + c_c (p-1) \eps  \kappa^{\frac{2sp}d} \left\|(f \wedge \kappa) \right\|^{- \frac{2s}d p}_{L^p} \int_{\R^d} \kappa^{p-1} (f-\kappa)_+(v) \langle v \rangle^{-k_0} \dv \\
                                                          &+ c_c (p-1) \bar{C}_\eps \kappa^{-\iota} \left(\int_{\R^d} \phik (f) (v) \dv \right)^{\tilde \beta}.
\end{align*}
We identify in the right-hand side a constant term, and we use that $(f\wedge \kappa)^p \le \phik (f)$ in order to estimate $\|(f \wedge \kappa) \|_{L^p}^{p \beta_\eps}$ by $\left( \int_{\R^d} \phik (f(t,v)) \dv \right)^{\beta_\eps}$. Then we use next Lemmas~\ref{l:lower} and \ref{l:lower-bis} and get the desired inequality for some constant $C_3$.
\end{proof}
\subsection{Propagation of approximate $L^p$ norms}

We consider the function $\phik$ introduced in Subsection~\ref{ss:phik} 
with $p \in \left(\frac{d}{d+\gamma+2s},\frac{d}{d+\gamma} \right)$ and $\kappa \ge 1$. 
\begin{lem}[Evolution of approximate $L^p$ norms]\label{l:evol-lp}
Let $f \colon (0,T) \times \R^d \to (0,+\infty)$ be a suitable weak solution of the Boltzmann equation.
  Let $p \in \left(\frac{d}{d+\gamma+2s},\frac{d}{d+\gamma} \right)$ and $\eps \in (0,1)$ and $\kappa \ge \kappa_0$ with
  $\kappa_0 \ge 1$ given by Lemma~\ref{l:lower-bis}. Then,
\[
  \frac{d}{dt} \int_{\R^d} \phik (f(t,v)) \dv \le \Cevol \left( \int_{\R^d} \phik (f(t,v)) \dv \right)^{\beta_\eps}
  + \Cevol \; \kappa^{-\iota} \left( \int_{\R^d} \phik (f(t,v)) \dv \right)^{\tilde \beta} 
  \text{ in } \mathcal{D}'((0,T)),
\]
where the constants $\Cevol, \beta, \tilde \beta, \iota$  depend only on $d,p,\gamma,s,\eps$ and the hydrodynamical bounds. In particular, it
does not depend on $\kappa \ge 1$. In particular, $\beta = 1 + \frac{2s}{p(d+\gamma+2s)-d} >1$.
\end{lem}
\begin{proof}
 Since $f$ is a suitable weak solution of the Boltzmann equation in the sense of Definition~\ref{d:suitable}, we conclude
that it satisfies in $\mathcal{D}'((0,T))$, 
\[
  \frac{d}{dt} \int_{\R^d} \phik (f(t,v)) \dv + \int_{\R^d} d_{\phik} (f,f') K (v,v') \dv \dv' \le c_c \int_{\R^d} \Phik (f) (f \ast_v |\cdot|^\gamma) \dv.
\]
We can apply next Lemma~\ref{e:key-estim} and get,
\begin{multline*}
  \frac{d}{dt} \int_{\R^d} \phik (f(t,v)) \dv  \\ \le   C_3 + C_3  \int_{\R^d} \phik (f(t,v)) \dv  + C_3 \left( \int_{\R^d} \phik (f(t,v)) \dv \right)^{\beta_\eps}
  + C_3 \kappa^{-\iota} \left(\int_{\R^d} \phik (f) (v) \dv \right)^{\tilde \beta}.
\end{multline*}
We use the pointwise lower bound on $f$ given by Lemma~\ref{l:pw-lower-bound}.
We recall that $\phik(f) \ge (f \wedge \kappa)^p$ with $\kappa \ge 1$ and $\ell_0 \in (0,1)$ and we get,
\[ \ell_0^{p} a_0 \le   \int_{\R^d} \phik (f(t,v)) \dv \]
where $a_0>0$ comes from Lemma~\ref{l:pw-lower-bound}. 
This yields the result with $\Cevol$ given by
\[ \Cevol = 3 C_3 \max ( \left(\ell_0^{p} |B_{R_0}|\right)^{-\beta_\eps},\left(\ell_0^{p} |B_{R_0}|\right)^{1-\beta_\eps}, 1). \qedhere \]
\end{proof}

\section{A  criterion for local-in-time boundedness}
\label{s:criterion}

In \cite{MR3551261}, the third author of this article proved the following result for classical solutions.
We will need it for suitable weak solutions. 
\begin{thm}[Criterion for boundedness] \label{t:prodi-serrin}
  Let $f \colon (0,T) \times \R^d \to [0,+\infty)$ be a suitable weak solution of the homogeneous Boltzmann equation
  for some initial data $\fin$. If for a.e. $t \in (0,T)$, we have
  \[  \| f(t,\cdot) \|_{L^p_k}  \le C_0  \]
  for $k$ sufficiently large and $p \in \left(\frac{d}{d+\gamma+2s},\frac{d}{d+\gamma} \right)$, then $f$ is essentially bounded in $(0,T) \times \R^d$: for a.e. $t \in (0,T)$,
  \[ \| f(t,\cdot)\|_{L^\infty(\R^d)} \le C (1+ t^{-\frac{d}{2sp}})\]
  for some constant $C$ depending only on $C_0$, the hydrodynamic bounds on $f$, dimension and the parameters $s,\gamma$ of the collision kernel. 
\end{thm}
The proof of Theorem~\ref{t:prodi-serrin} consists in  studying the propagation of the approximate $L^p$ norm of $(f-a(t))_+$,
\[ \varphi (f) =  \phik ((f-a)_+) \]
where we recall that
\[ \phik (r) = (r \wedge \kappa)^p + q \kappa^{p-1} (r-\kappa)_+.\]
Straightforward computations yield,
\begin{equation}
  \label{e:phi-dphi-k}
  \begin{cases}
  \Phi (r) &= \Phi_\kappa ((r-a)_+) + a \phik' ((r-a)_+), \\
  d_\varphi (r,s) & = d_{\phik} ((r-a)_+, (s-a)_+) + \phik' ((r-a)_+)(a-s)_+
  \end{cases}
\end{equation}
where $\Phi_\kappa$ and $d_{\phik}$ are $\Phi$ and $d_\varphi$ for $\varphi = \phik$.

\subsection{Time-dependent truncation}

In order to establish this result, we first remark that the entropy inequality \eqref{e:generalentropy} can be obtained for a function $a=a (t)$. 
\begin{lem}[Time-dependent truncation] \label{l:phit}
  Let  $a \colon (0,T) \to (0,+\infty)$ be $C^1$ and  $f \colon (0,T) \times \R^d \to [0,+\infty)$ be a suitable weak solution of the homogeneous Boltzmann equation such that for a.e. $t \in (0,T)$, 
  \[  \| f(t,\cdot) \|_{L^p_k}  \le C_0  \]
  for $k$ sufficiently large and $p \in \left(\frac{d}{d+\gamma+2s},\frac{d}{d+\gamma} \right)$. Then, for $\varphi (t,r) = (r-a(t))_+$, we have 
  \begin{multline*} 
  \frac{d}{dt} \int_{\R^d} \varphi (t,f(t,v)) \dv + \iint_{\R^d \times \R^d} d_\varphi (f,f') K (v,v') \dv \dv' \\
  \le c_c \int_{\R^d} \Phi (t,f) (f \ast_v |\cdot|^\gamma) \dv - a'(t) \int_{\R^d} \varphi'_a (f(t,v)) \dv \quad \text{ in } \mathcal{D}'((0,T)).
  \end{multline*}
\end{lem}
\begin{proof}
We follow ideas from Kruzhkov's original idea to get the $L^1$-contraction principle for scalar conservation laws. 
It is enough to consider $a = a(s)$ and a test-function $\phi_{s,\eps} (t,v) = \theta_\eps (t-s) \phi (t,v)$ and integrate over $s$.
We then use that $\partial_t (\theta_\eps (t-s)) = - \partial_s (\theta_\eps (t-s))$ and integrate by parts with respect to $s$.
Letting $\eps \to 0$ yields the result.
\end{proof}

\subsection{Coercivity estimates}

\begin{lem}[A lower bound for the dissipation of the truncated approximate norm]\label{e:coer-enplus}
Assume that $f$ satisfies: \(  \| f \|_{L^p_k}  \le C_0 . \)
Then
\[
  \iint_{\R^d \times \R^d} \phik' ((f-a)_+)(a-f')_+ K (v,v') \dv \dv' 
  \ge \bar c_I  a^{1+\frac{2sp}d} \int_{\R^d}  \phik' ((f-a)_+) \langle v \rangle^{\bar \omega_1} \dv 
\]
with $\bar \omega_1 = \gamma +2s + \frac{2s}{d}(kp -1)$, 
for some constant $\bar c_I$ only depending $C_0$, $k$, $p$, $\lambda_0$, $s$ and $d$. 
\end{lem}
\begin{proof}
\begin{align*}
  \iint_{\R^d \times \R^d} \phik' ((f-a)_+)(a-f')_+ &K (v,v') \dv \dv' \\
  &\ge \frac{a}2 \iint_{\{f' < a/2\}}  \phik' ((f-a)_+)   K (v,v') \dv' \dv \\
&\ge \frac{a}2 \int_{\R^d}  \phik' ((f-a)_+) \left\{ \int_{\{f' < a/2\} \cap C(v) \cap T_R (v)} K (v,v') \dv' \right\} \dv
\end{align*}
where we recall that $C(v)$ denotes the non-degeneracy cone from Lemma~\ref{l:cone} and we consider,
\[
  T_R (v) = \begin{cases}
              B_{R/2} (v) & \text{ if } |v| \ge R, \\ B_{3R} (v) \setminus B_{2R} (v) & \text{ if } |v| \le R.
            \end{cases}
\]
We remark that for $v' \in T_R(v)$, we have $\langle v \rangle \le 2 \langle v' \rangle$. 

Here, we choose $R$ so that $f' < a/2$ in half of $B_R(v) \cap C(v)$.  We achieve it by taking
\[ \langle v \rangle^{-1} R^d = C_R \frac{\|f\|_{L^p_k}^p}{\langle v \rangle^{k p} a^p} \]
for some  constant $C>0$ large enough (only depending on $k,p$ and $\lambda_0$ from Lemma~\ref{l:cone}).
Indeed, we use once again Chebyshev's inequality to write,
\begin{align*}
  |\{f' \ge (a/2) \} \cap C(v) \cap T_R (v) | &\le \frac{2^{(k+1)p}}{a^p \langle v \rangle^{kp}} \int_{T_R (v)} f^p (v') \langle
                                                v' \rangle^{kp} \dv'\\
  & \le \frac{2^{(k+1)p}}{C_R} \frac{R^d}{\langle v \rangle}.
\end{align*}
Using Lemma~\ref{l:cone} (see also Remark~\ref{r:tr} for the case $|v|\le R$), this implies, 
\begin{align*}
  |\{f' < (a/2) \} \cap C(v) \cap T_R (v) | &\ge \left(\lambda_0 - \frac{2^{(k+1)p}}{C_R}\right) \frac{R^d}{\langle v \rangle} \\
  & \ge \frac{\lambda_0}2 \frac{R^d}{\langle v \rangle}
\end{align*}
for $C_R = 2^{k+2} p \lambda_0^{-1}$. 
We remark that \[R^d \le  \frac{C_R C_0^p}{\langle v \rangle^{kp-1}  a^p}\]
where $C_0$ comes from the assumption of the theorem.

With this choice of $R$, we get the following inequality
\begin{align*}
  \iint_{\R^d \times \R^d} \phik' ((f-a)_+)(a-f')_+ &K (v,v') \dv \dv' \\ 
      & \geq \frac{\lambda_0^2 }{4} a \int_{\R^d} \phik' ((f-a)_+)   \langle v \rangle^{\gamma+2s} R^{-2s} \dv \\
\intertext{we now use the upper bound for $R$,}
                                                    & \geq \frac{\lambda_0^2}{4} a \int_{\R^d} \phik' ((f-a)_+) \langle v \rangle^{\gamma+2s} (C_R C_0^p)^{\frac{-2s}d}
        \langle v \rangle^{\frac{2s}d (kp -1)} a^{\frac{2s}d p}\dv \\
    &\geq \bar c_I  a^{1+2sp/d} \int_{\R^d}  \phik' ((f-a)_+) \langle v \rangle^{\gamma+2s+\frac{2s}d (kp-1)} \dv
\end{align*}
with $\bar c_I=\frac{\lambda_0^2}{4} (C_R C_0^p)^{\frac{-2s}d}$. 
\end{proof}

\subsection{Proof of conditional boundedness}

\begin{proof}[Proof of Theorem \ref{t:prodi-serrin}]
We consider the decreasing function,
\[ a(t) = C_a(1+t^{-d/(2sp)}),\]
(for some constant $C_a>1$ to be determined) and the function $m(t)$ for $t>0$ defined by,
\[ m(t) := e^{C_m t}\int_{\R^d} \phik ((f(t,v) - a(t))_+) \dv . \]
Since $a(t) \to +\infty$ as $t \to 0$, we can make $m(t)$ arbitrarily small near $t = 0$. We prove Theorem \ref{t:prodi-serrin} after we show that $m(t)$ is monotone decreasing in time. As a matter of fact, we prove something slightly weaker,
\[ \frac{\dd}{\dt} m (t) \le C_m \kappa^{-\iota} \]
for some $\iota >0$. Given $t_1>0$, we pick $t_0 \in (0,t_1)$, we integrate between $t_0$ and $t_1$ and we let $t_0 \to 0$ and $\kappa \to +\infty$: this yields $m(t_1)=0$. 

Using Lemma \ref{l:phit} after integrating with respect to $a$ against $\varphi''(a)$ (see Formula~\eqref{e:convex-rep}), we see that in order to prove that $m(t)$ is decreasing, we must verify that for all $t \in (0,T)$,
\begin{multline*} 
  \iint_{\R^d \times \R^d} d_\varphi (f,f') K (v,v') \dv \dv' \\
  \geq c_c \int_{\R^d} \Phi (f) (f \ast_v |\cdot|^\gamma) \dv - a' \int_{\R^d}\phik' ((f-a)_+) \dv - C_m \int_{\R^d} \varphi (f) \dv .
\end{multline*}
In view of \eqref{e:phi-dphi-k}, this is equivalent to
\begin{equation} \label{e:aim}
  \begin{aligned}
  \iint_{\R^d \times \R^d} & d_{\phik} ((f-a)_+,(f'-a)_+) K (v,v') \dv \dv' + \iint_{\R^d \times \R^d} \phik' ((f-a)_+)(a-f')_+ K (v,v') \dv \dv'\\
                                                              &\geq c_c \int_{\R^d} \Phi_{\kappa} ((f-a)_+) (f \ast_v |\cdot|^\gamma) \dv 
    +c_c a \int_{\R^d} \phik' ((f-a)_+) (f \ast_v |\cdot|^\gamma) \dv \\
  &-a' \int_{\R^d} \phik' ((f-a)_+) \dv - C_m \int_{\R^d} \phik ((f-a)_+) \dv - C_m \kappa^{-\iota}.
  \end{aligned}
\end{equation}

\paragraph{Lower bounds for (LHS).}  Let (LHS) denote the left-hand side of \eqref{e:aim}.
Lemmas~\ref{l:lower}, \ref{l:lower-bis} and \ref{e:coer-enplus} give a lower bound for each term of this quantity. Taking their mean, we get 
\begin{equation}
  \label{e:i}
\begin{aligned}
  \mathrm{(LHS)} \geq &c_{1,D} \left\| \left((f-a)_+ \wedge \kappa \right)^p \right\|_{L^{p_0}_{k_0}} \\
  +& c_{2,D}  \kappa^{\frac{2sp }d} \| ((f-a)_+ \wedge \kappa)\|_{L^p}^{-\frac{2sp}d} \int_{\R^d} \kappa^{p-1}(f-a-\kappa)_+ (v) \langle v \rangle^{k_0} \dv\\
  +& \bar c_I a^{1+\frac{2sp}d} \left( \int_{\R^d} \phik' ((f-a)_+) \langle v \rangle^{\bar \omega_1} \dv \right)  - C_{1,D} \int_{\R^d} \varphi(f)  \dv
\end{aligned}
\end{equation}
with $\bar \omega_1$ arbitrarily large (thanks to $k$).
\bigskip

We need appropriate bounds for
\begin{align*}
 \mathrm{(i)} &=  \int_{\R^d} \Phi_\kappa((f-a)_+) \times (f \ast |\cdot|^\gamma) \dv,\\
 \mathrm{(ii)} &=   a \int_{\R^d} \phik' ((f-a)_+) \times ((f \wedge a) \ast |\cdot|^\gamma) \dv, \\
  \mathrm{(iii)} &=   a \int_{\R^d} \phik'((f-a)_+)\times ((f-a)_+ \wedge \kappa \ast |\cdot|^\gamma) \dv, \\
  \mathrm{(iv)} &=   a \int_{\R^d} \phik'((f-a)_+)\times ((f-a-\kappa)_+  \ast |\cdot|^\gamma) \dv.
\end{align*}
 
\paragraph{Upper bound for (i).}
We recall that $\Phi_\kappa (r) = (p-1) (r \wedge \kappa)^p$. In particular, 
\[ \mathrm{(i)} =  (p-1)\int_{\R^d} \left((f-a)_+ \wedge \kappa\right)^p \times (f \ast |\cdot|^\gamma) \dv.\]
Since $k >0$, we have in particular, if $F$ denotes $((f-a)_+\wedge \kappa)^p$,
\begin{align*}
\mathrm{(i)}  & \leq c_c(q-1) \|f\|_{L^p} \|F\|_{L^r} \\
       & \leq c_c(q-1) \|f\|_{L^p_k} \|F\|_{L^r} 
\end{align*}
where $1/r + 1/p = 2 + \gamma/d$. If we consider $r_p = p r$, we can write
\[\mathrm{(i)} \leq c_c q C_0 \| (f-a)_+ \wedge \kappa\|_{L^{r_p}}^p. \]
We now use Lemma~\ref{l:interpol-again} with $F=(f-a)_+ \wedge \kappa$ and get for all $\eps >0$,
\begin{equation} \label{e:iic}
  \mathrm{(i)} \le C_\eps\int_{\R^d} ((f-a)_+\wedge \kappa)^{p-1} \langle v \rangle^{k_p}  \dv   + \eps \left\| ((f-a)_+\wedge \kappa)^p\right\|_{L^{p_0}_{k_0}}  
\end{equation}
where $k_p$ and $C_\eps^0$ comes from the lemma (after rescaling).

\paragraph{Upper bound for  (ii).}
For $r>1$ such that $\frac1{r}- \frac{\gamma}d =1$, we have 
\begin{align*}
  \| (f \wedge a) \ast |\cdot|^\gamma \|_{L^\infty}& \le \| (f\wedge a)\|_{L^r} \\
  & \le \| (f\wedge a)\|_{L^p}^\theta \| (f\wedge a)\|_{L^\infty}^{1-\theta} 
\end{align*}
with $\theta \in (0,1)$ such that $\frac{\theta}p + \frac{1-\theta}{\infty} = \frac1{r}$.
We have to check that $p<r$. This is equivalent to $p < \frac{d}{d+\gamma}$.
Moreover, we have
\begin{equation}
  \label{e:theta}
  (1-\theta) = 1 - p \left(1 + \frac{\gamma}d \right) < \frac{2s p} d.
\end{equation}
Since $a \ge 1$, we thus proved that
\begin{equation}
  \label{e:iia}
\mathrm{(ii)}  \leq \Ciia a^{1+(1-\theta)} \int_{\R^d}  \phik'((f-a)_+) \dv
\end{equation}
with $\theta \in (0,1)$ such that \eqref{e:theta} holds true.

\paragraph{Upper bound for (iii).}
Let $F$ denote $(f-a)_+ \wedge \kappa$ again. We use first H\"older's inequality,
\begin{align*}
  \mathrm{(iii)} &= c_c p  \int_{\R^d} (f\wedge a) F^{p-1} \times (F \ast |\cdot|^\gamma) \dv, \\
                 & \le c_c p \| f\wedge a\|_{L^p} \left\| F^{p-1} \times (F \ast |\cdot|^\gamma) \right\|_{L^{p'}}
\end{align*}
 with $\frac1p + \frac1p' =1$.
The first norm can be estimated thanks to the assumption. As far as the second norm is concerned, we consider
$\frac1{r_p} = \frac1p \left( 1 + \frac{\gamma}d + \frac1{p'} \right)$ and $p_2$ such that $\frac1{p_2} + \frac{p-1}{r_p} = \frac1{p'}$. The range of values $p$ corresponds to the condition $r_p \in (p,pp_0)$.
In particular, we have $1 +\frac1{p_2} = \frac1{r_p} - \frac{\gamma}d$.
We then use H\"older's inequality and then weak Young's convolution inequality to get,
\begin{align*}
  \mathrm{(iii)}  & \le c_c p C_0 \left\| F^{p-1}  \right\|_{L^{\frac{r_p}{p-1}}} \left\| (F \ast |\cdot|^\gamma) \right\|_{L^{p_2}} \\
                & \le c_c p C_0 \left\| F^{p-1}  \right\|_{L^{\frac{r_p}{p-1}}} \left\| F \right\|_{L^{r_p}}  \\
  & = c_c p C_0  \left\| F\right\|_{L^{r_p}}^p.
\end{align*}
We now apply Lemma~\ref{l:interpol-again} and get for all $\eps >0$,
\begin{equation}\label{e:iib}
  \mathrm{(iii)}  \le C_\eps \int_{\R^d} ((f-a)_+ \wedge \kappa)^{p-1} \langle v \rangle^{k_p} \dv   + \eps \left\|((f-a)_+ \wedge \kappa)^p \right\|_{L^{p_0}_{k_0}}  
\end{equation}
for some $k_q >0$ and some constant $C_\eps$ depending on $\eps$, $c_c$, $M_0$ and $q$. 

\paragraph{Upper bound for (iv).}
This term is estimated thanks to Lemma~\ref{l:error-II} applied to $F = (f-a)_+$. 

\begin{align}
\nonumber  \mathrm{(iv)} &=   a \int_{\R^d} \phik'(F)\times ((F-\kappa)_+  \ast |\cdot|^\gamma) \dv \\
\nonumber  &  \le \eps  \kappa^{\frac{2sp}d} \left\|(F \wedge \kappa) \right\|^{- \frac{2s}d p}_{L^p} \int_{\R^d} \kappa^{p-1} (F-\kappa)_+(v) \langle v \rangle^{k_0} \dv \\
\label{e:iv-estim}    &+  \tilde{C}_\eps \kappa^{-\iota} \left(\int_{\R^d} \phik (F)  \dv \right)^{\bar{\beta}}.
\end{align}

\paragraph{Conclusion.}
Picking $k$ large enough so that $\bar \omega_1 > k_p>0$ from \eqref{e:iib} and \eqref{e:iic}, we combine \eqref{e:i}, \eqref{e:iic}, \eqref{e:iia}, \eqref{e:iib}, and \eqref{e:iv-estim},
\begin{align*}
  \mathrm{(LHS)} - \mathrm{(RHS)} \geq &  (c_{1,D}-2\eps) \|\varphi(f)\|_{L^{p_0}_{k_0}} \\
  + &(c_{2,D}-\eps)  \kappa^{\frac{2sp }d} \| ((f-a)_+ \wedge \kappa)\|_{L^p}^{-\frac{2sp}d} \int_{\R^d} \kappa^{p-1}(f-a-\kappa)_+ (v) \langle v \rangle^{k_0} \dv\\
   -& \tilde C_\eps \kappa^{-\iota} \left(\int_{\R^d} \varphi (f) \dv \right)^{\bar{\beta}} -C_{1,D} \int_{\R^d} \varphi (f) \dv - \tilde{C}_\eps  C_0^{\bar \beta} \kappa^{-\iota}\\
   +& \left(q a'(t) - 2 C_\eps - \Ciia a(t)^{1+ (1-\theta)} + \bar c_I a(t)^{1+\frac{2sp}d} \right)  \int_{\R^d} ((f(t,v) - a(t))_+\wedge \kappa)^{q-1} \langle v \rangle^{\bar \omega_1} \dv.    
\end{align*}

Choosing $\eps = \min( c_{1,D}/2, c_{2,D})$, we are left with checking that we can choose $C_a$ large enough so that $a = C_a (1+ t^{-\frac{d}{2sp}})$ satisfies
\[
  q a' (t) \geq  2 C_\eps + \Ciia a(t)^{1+ 2sp/d - \delta} - \bar c_I a(t)^{1+2sp/d} 
\]
where $\delta = \frac{2sp}d - (1-\theta) \in (0, 2sp/d)$, see \eqref{e:theta}.
We rewrite this inequality as,
\[
  - a' (t) \leq  \frac1{q}\left(\bar c_I  - 2 C_\eps a(t)^{-1-2sp/d} - \Ciia a(t)^{- \delta} \right) a(t)^{1+2sp/d} .
\]
Since $a \ge C_a$, this inequality holds true as soon as,
\[
  - a' (t) \leq  \frac1{q}\left(\bar c_I  - 2 C_\eps C_a^{-1-2sp/d} - \Ciia C_a^{- \delta} \right) a(t)^{1+2sp/d} .
\]
We first pick $C_a$ large enough so that 
\[ \frac1{q}\left(\bar c_I  - 2 C_\eps C_a^{-1-2sp/d} - \Ciia C_a^{- \delta} \right)  \ge \frac{\bar c_I}{2q}\]
and we are left with checking,
\[
  - a' (t) \leq  \frac{\bar c_I}{2q} a(t)^{1+2sp/d} .
\]
In view of the definition of $a$, this is equivalent to,
\[
  \frac{d}{2sp} t^{-\frac{d}{2sp}-1}  \leq  \frac{\bar c_I C_a^{\frac{2sp}{d}}}{2q} \left(1+t^{-\frac{d}{2sp}}\right)^{1+2sp/d} .
\]
It is now sufficient to pick $C_a$ large enough so that
\[ \frac{\bar c_I C_a^{\frac{2sp}{d}}}{2q} \ge \frac{d}{2sp}.\]
This achieves the proof of the theorem. 
\end{proof}

\section{Partial regularity}
\label{s:partial}

The proof of the main result relies on the propagation of $L^p$ norms for some well chosen $p$.
Such a propagation result is obtained by approximation: we consider a convex function $\phik(r)$ that approximates $r^p$
and can be used in the definition of suitable weak solutions (see Definition~\ref{d:suitable}).

\subsection{An improved criterion for boundedness}

\begin{lem}[Small approximate entropy dissipation implies boundedness] \label{l:improved}
Let $T \in (0,1)$ and $f \colon (0,T) \times \R^d \to (0,+\infty)$ be a suitable weak solution of the Boltzmann equation.
Let $p \in \left(\frac{d}{d+\gamma+2s},\frac{d}{d+\gamma} \right)$ and $\eps \in (0,1)$ and $\kappa \ge \max(\kappa_0,\kappa_1)$ with
$\kappa_0$ given by Lemma~\ref{l:lower-bis} above and $\kappa_1$ given by Lemma~\ref{l:technical} below. Then,
\[
  \frac1T \int_0^T \left( \int_{\R^d} \phik (f(t,v)) \dv \right)^{\frac1p} \dt \le T^{-\frac1{p \alpha}} \eta \quad \Rightarrow  \quad  \| f(t,\cdot)\|_{L^\infty((T/2,T) \times \R^d)} \le \Cic 
\]
with $\alpha = \beta_\eps - 1>0$. We recall that $\beta_\eps \to \beta$ as $\eps \to 0$ with $\beta$ given by \eqref{e:beta}. The constants $\eta$ and $\Cic$ only depend on $d,p,s,\gamma,\eps,p$.  
\end{lem}

This lemma is an immediate consequence of Lemma~\ref{l:evol-lp}, Theorem~\ref{t:prodi-serrin} and the following technical result --  applied to $X (t) =  \left( \int_{\R^d} \phik (f(t,v)) \dv \right)^{\frac1p}$.
\begin{lem} \label{l:technical}
  Let $\kappa \ge 1$, $\iota >0$, $p\alpha > 1$, $\delta >0$ and $X \colon (0,T) \to (0,+\infty)$ be essentially bounded
  and such that 
  \[ \dot X \le C X^{1+p\alpha} + D \kappa^{-\iota} X^{1+p\alpha + \delta} \quad \text{ in } \mathcal{D}'((0,T)) \]
  for some constants $C,D \ge 0$. 
  Then for $\kappa \ge \kappa_1$,
  \[ \frac1T \int_0^T X (t) \dt \le T^{-\frac1{p\alpha}} \eta \Rightarrow \|X\|_{L^\infty(T/3,T)} \le M T^{\frac1{p\alpha}}\]
with $\eta$ and $M$ only depending on $C,D$ and $\alpha$ and $\kappa_1$ depending on $C,D,p\alpha$ and $\delta$. 
\end{lem}
\begin{proof}
  We reduce to the case $T=1$ by scaling: $Y(t) = T^{\frac1{p\alpha}} X (T t)$.

Then we first deal with the case $D=0$. Recall that $T=1$. By considering $Y(t) = c_0 X(t)$, we have
\[ \dot Y \le \bar C Y^{1+p\alpha}  \quad \text{ in } \mathcal{D}'((0,T)) \]
with $\bar C = C / c_0^p\alpha$. We then choose $c_0$ such that $p\alpha \bar C \le 1/2$. We simply pick $c_0 = (2p\alpha C)^{\frac1{p\alpha}}$.

We then consider the set $\mathcal{S} = \{ t \in (0,1): Y (t) \ge 1 \}$. Chebyshev's inequality ensures that $|S| \le c_0 \eta$
(recall that $\int_0^1 Y (t) \dt \le \eta$). In particular, if we choose $c_0 \eta < 1/3$, we know that $(0,1/3) \setminus \mathcal{S}$
has a positive measure. We pick $\eta = (4 c_0)^{-1}$.

For $t \in (0,1/3) \setminus S$ and $h \in (0,1-t)$, we have
\[ Y^{-p\alpha} (t+h) \ge 1 - \bar C p\alpha h \ge 1 - \bar C p\alpha \ge 1/2.\]
We conclude that for a.e. $h \in (0,1-t)$, we have
\[ Y (t+h) \le 2^{p\alpha}.\]
This implies that
\[ \|X \|_{L^\infty (1/3,1)} \le (2 p\alpha C)^{-\frac1{p\alpha}} 2^{p\alpha}.\]

We finally treat the general case. We consider $M_0 = \kappa^{\frac\iota\delta} D ^{-\frac1\delta}$. In particular,
$D \kappa^{-\iota} M_0^\delta =1$. We then consider $Y (t) = \min (X(t),M_0)$. Since $Y (t) \le M_0$, it satisfies
\[ \dot Y \le \bar C Y^{1+p\alpha} \quad \text{ in } \mathcal{D}'((0,T)) \]
with $\bar C = C + D$. The previous case implies that
\[ \frac1T \int_0^T X(t) \dt \le T^{-\frac1{p\alpha}} \eta \Rightarrow \| Y \|_{L^\infty (T/3,T)} \le M \]
with
 $\eta = (4 (2p\alpha (C+D))^{\frac1{p\alpha}})^{-1}$ and $M = (2 p\alpha (C+D))^{-\frac1{p\alpha}} 2^{p\alpha}$. 
 We finally pick $\kappa_1 \ge 1$ such that for $\kappa \ge \kappa_1$, we have $M_0 > M$. 
 We conclude that $\| X \|_{L^\infty (T/3,T)} \le M$ in the general case. 
\end{proof}

\subsection{Proof of the main theorem}

\begin{proof}[Proof of Theorem~\ref{t:main}]
The estimate of the set $S$ of singular times of a suitable weak solution $f$ of the homogeneous Boltzmann equation  derives from  Lemma~\ref{l:improved} through a covering argument.

We pick $r_0>0$. For any singular time $\tau$ in the sense of Definition~\ref{defi:singular}, Lemma~\ref{l:improved} implies that
for all $\kappa \ge \max(\kappa_0,\kappa_1)$, and all $r \in (0,r_0)$,
\[
 \int_{\tau-r}^\tau \int_{\R^d} \phik (f(t,v)) \dv \dt > \eta \, r^{1-\frac1 {p\alpha}}.
\]
Since $\kappa \mapsto \phik (f)$ is non-negative and monotone increasing and $\phik (f) \to f^p$ as $\kappa \to \infty$,
Beppo-Levi's theorem implies that
\begin{equation}\label{e:deltaT}
 \int_{\tau-r}^\tau \int_{\R^d} f^p(t,v) \dv \dt > \eta \, r^{1-\frac1 {p\alpha}}.
\end{equation}

Hence,  $S$ is covered by intervals of the form $(\tau-r,\tau+r)$ such that \eqref{e:deltaT} holds true,
\[
 S \subseteq \bigcup_{\tau \in S} (\tau-r,\tau+r).
\]
By Vitali's covering lemma, there exists a countable subcollection of disjoint intervals $(\tau_i-r_i,\tau_i+r_i)$ such that
\[
S \subseteq \bigcup_{i =1}^{+\infty} (\tau_i-5r_i,\tau_i+5r_i).
\]
In particular 
\[
\sum_{i =1}^{+\infty} \int_{\tau_i-r_i}^{\tau_i} \int_{\R^3} f^p (t,v) \dv \dt > \eta \sum_{i =1}^{+\infty}  r_i^{1-\frac1{p\alpha}},
\]
and since all the intervals $(\tau_i-r_i,\tau_i+r_i)$ are disjoint, 
\begin{align*}
 \eta \sum_{i =1}^{+\infty}  r_i^{1-\frac1{p\alpha}} & \le \sum_{i =1}^{+\infty} \int_{\tau_i-r_i}^{\tau_i} \int_{\R^3} f^p (t,v) \dv \dt \\
& \le \int_0^T \int_{\R^3} f^p (t,v) \dv \dt .
\end{align*}

Since $r_0$ is arbitrarily small, this implies that the Hausdorff dimension of $S$ is smaller than $1-\frac1{p\alpha}$.

Since $\alpha$ is arbitrarily close to $\alpha_0 = \frac{2s}{p(d+\gamma+2s)-d}$ (see Lemma~\ref{l:improved}) and $p$ is arbitrarily close to $p_0$ (see the entropy dissipation
 estimate from \eqref{t:entropy} and Theorem~\ref{t:moments} about propagation of moments), we get
that the Hausdorff dimension is smaller than
\[ 1 - \frac1{\alpha_0 p_0} = -\frac{\gamma+2s}{2s} \ge 0. \qedhere \]
\end{proof}

\appendix

\section{Construction of suitable weak solutions}
\label{s:construction}

This section is devoted to the construction of suitable weak solutions.
\begin{prop}[Existence of suitable weak solutions]\label{p:existence}
  Let $B (z,\cos \theta) = |z|^\gamma b(\cos \theta)$ with $\gamma > -d$ and $b$ satisfying \eqref{e:b} for some $s \in (0,1)$.
  Let $\fin \in L^1_2 (\R^d)$ be positive a.e. and such that $\fin \ln \fin \in L^1(\R^d)$.
  There exists a suitable weak solution of the Boltzmann equation associated to the initial data $\fin$. 
\end{prop}

In order to construct suitable weak solutions, we need to be able to derive the family of inequalities
associated with a family of convex functions. In order to get them, we need almost sure convergence to
identify weak limits. This is the reason why we follow Arkeryd by reducing to bounded kernels but we
depart from his work by considering an artificial vanishing viscosity. 

\subsection{Estimates for H-solutions}

In order to to construct suitable weak solutions, we recall some facts from \cite{MR1650006}.
Since we work in any dimension $d \ge 2$, we provide proofs. 

\paragraph{Entropy dissipation.}

The entropy dissipation plays a key role in the theory of H-solutions. Using symmetries of pre- and post-collisional velocities, we easily get the following formula,
\begin{multline} \label{e:dissip-weak}
  \int_{\R^d} Q_B (f,f) \ln f \dv \\
      = \frac14\iiint_{\R^d \times \R^d \times \Sd} (f(w')f(v') - f(w)f(v)) \ln \frac{f(w')f(v')}{f (w) f(v)}
    B (|v-w|,\cos \theta) \dv \dw \dsigma \ge 0.
\end{multline}
C. Villani observed \cite{MR1650006} that thanks to an elementary convex inequality, the $L^2$-norm of a well-chosen function is controlled. 
\begin{lem}[Entropy dissipation estimate] \label{l:dissip}
We consider $B(z,\cos \theta) = \Psi (|z|) b(\cos \theta)$. 
  If $\mathcal{F} (v,w)$ denotes $\sqrt{ f(v) f(w) \Psi (|v-w|)}$, then
\begin{align*}  &\iiint_{\R^d \times \R^d \times \Sd} \left(\mathcal{F}(v',w') - \mathcal{F}(v,w) \right)^2 b (\cos \theta) \dv \dw \dsigma\\
  & = \iiint_{\R^d \times \R^d \times \Sd} (\sqrt{f(w')f(v')} - \sqrt{f(w)f(v)})^2  \Psi (|v-w|) b (\cos \theta)  \dv \dw \dsigma \\
  & \le \int_{\R^d} Q_{B} (f,f) \ln f \dv .
\end{align*}
\end{lem}
\begin{proof}
  This is an immediate consequence of \eqref{e:dissip-weak} and the inequality
  $(a-b) (\ln a-\ln b) \ge 4 (\sqrt{a}-\sqrt{b})^2$ that holds true for $a,b>0$.
\end{proof}
The estimate given by the previous lemma allows us to understand
the collision operator in the following weak sense.
\begin{lem}[Weak form of the collision operator]\label{l:operator-weak}
  We consider $B(z,\cos \theta) = \Psi (|z|) b(\cos \theta)$ with
  \[ \Psi (|z|) \le |z|^\gamma \text{ and } \gamma \in [-4s,-2s] \text{ and }
    C_b = \int_{\Sd} (1-\cos \theta)^{s+\eps} b (\cos \theta) \dsigma <+\infty \]
  with $\eps>0$ such that $s +\eps <1$. 
  Let $f \colon \R^d \to \R$ such that \eqref{e:hydro} holds true for $\fin =f$. 
    If $\mathcal{F} (v,w)$ denotes $\sqrt{ f(v) f(w) \Psi (|v-w|)}$, then for all $\varphi \in \mathcal{D} (\R^d)$, 
\begin{align*}
  \int_{\R^d} Q_{B} (f,f) \varphi \dv = & \frac14 \int_{\R^d \times \R^d \times \Sd} (\mathcal{F}(v',w') - \mathcal{F}(v,w))\sqrt{b (\cos \theta)} \\
  & \times  (\mathcal{F}(v,w) + \mathcal{F}(v',w')) \delta \varphi (v,w,v',w')  \sqrt{b (\cos \theta)}  \dv \dw \dsigma
\end{align*}
with $\delta \varphi (v,w,v',w') = \varphi (v)+\varphi (w) - \varphi (v') - \varphi (w')$ and
\begin{multline*}
  \int_{\R^d \times \R^d \times \Sd} (\mathcal{F}(v,w) + \mathcal{F}(v',w'))^2 (\delta \varphi)^2 (v,w,v',w')  b (\cos \theta)  \dv \dw \dsigma \\
  \le 4^3 C_b  (M_0 + E_0)^2 ( \|D^2_v \varphi\|_{L^\infty} + \|\varphi\|_{L^\infty})^2.
\end{multline*}
\end{lem}
\begin{proof}
The first equality is straightforward. As far as the $L^2$-norm is concerned, we write
Moreover, it can be shown through a Taylor expansion \cite{MR1650006} that
\[ |\delta \varphi (v,w,v',w')| \le 2 \|D^2_v \varphi\|_{L^\infty} |v-w|^2 (1-\cos \theta).\]
Let us give details. We follow \cite[p.~286]{MR1650006} by considering $x_1= \frac{v+w}2$ and $r = \frac{|v-w|}2$
and $k = \frac{v-w}{|v-w|}$ and $\sigma = \frac{v'-w'}{|v'-w'|}$, so that
\begin{align*}
  \delta \varphi (v,w,v',w') &= \bigg[ \varphi (x_1 + r \sigma) - \varphi (x_1+r k) \bigg] + \bigg[ \varphi (x_1-r\sigma)  - \varphi (x_1 -rk) \bigg]\\
                             &= \int_0^1 \bigg[ \nabla \varphi (x_1 + r (\tau \sigma + (1-\tau) k)) - \nabla \varphi (x_1 - r (\tau \sigma + (1-\tau)k))   \bigg] \cdot r(\sigma-k) \dd \tau \\
  & =  \int_0^1 \left[\int_{-1}^1 D^2 \varphi (x_1 + \bar \tau r (\tau \sigma + (1-\tau) k)) \dd \bar \tau \right] r (\tau \sigma + (1-\tau) k) \cdot r(\sigma-k) \dd \tau.
\end{align*}
This equality implies that 
\begin{align*}
  |\delta \varphi (v,w,v',w')|& \le 2 \| D^2 \varphi\|_{L^\infty} r^2 \int_0^1  |\tau \sigma + (1-\tau) k| |\sigma-k| \dd \tau\\
                              & \le 2 \| D^2 \varphi\|_{L^\infty} r^2 |\sigma - k | \\
  & \le 2 \| D^2 \varphi\|_{L^\infty} r^2 \sqrt{2(1-\cos \theta)}.
\end{align*}
We used the fact that $|\sigma -k|^2 = 2 - 2 \sigma \cdot k = 2 (1-\cos \theta)$. 

Since $\varphi$ is also bounded, we get
\[
  |\delta \varphi (v,w,v',w')| \le 4 (\|D^2_v \varphi\|_{L^\infty} + \|\varphi\|_{L^\infty}) \min (  |v-w|^2 \sqrt{2(1-\cos \theta)}, 1 ).
\]
This allows us to get the following estimate,
\begin{align*}
& \int_{\R^d \times \R^d \times \Sd} (\mathcal{F}(v,w) + \mathcal{F}(v',w'))^2 (\delta \varphi)^2 (v,w,v',w') b (\cos \theta)
     \dv \dw \dsigma \\
  &\le   C_\varphi^2 \int_{\R^d \times \R^d \times \Sd} \mathcal{F}^2(v,w) \min( |v-w|^4 2(1-\cos \theta),1)  b (\cos \theta)
     \dv \dw \dsigma \\
  \intertext{with $C_\varphi = 4 \|D^2_v \varphi\|_{L^\infty} + 4 \|\varphi\|_{L^\infty}$,}
  & \le  C_\Psi C_\varphi^2   \int_{\R^d \times \R^d} f(v)f(w) \left\{ \int_{\Sd}  \min( |v-w|^4 2(1-\cos \theta),1) b (\cos \theta) \dsigma \right\} |v-w|^\gamma      \dv \dw \\
\intertext{with $C_\Psi$ such that $\Psi_n (z) \le C_\Psi |z|^\gamma$. We now distinguish the cases $|v-w|\le 1$ and $|v-w|\ge 1$,}
  & \le    C_\Psi C_\varphi^2   \int_{\{|v-w| \le 1\}} f(v)f(w) \left\{ \int_{\Sd} 2(1-\cos \theta) b (\cos \theta) \dsigma \right\} |v-w|^{\gamma+4}      \dv \dw \\
  & +  C_\Psi C_\varphi^2   \int_{\{|v-w| \ge 1\}} f(v)f(w) \left\{ \int_{\Sd} 2^{s+\eps} (1-\cos \theta)^{s+\eps} b (\cos \theta) \dsigma \right\} |v-w|^{\gamma+4s + 2\eps}      \dv \dw \\
  \intertext{for $\eps>0$ such that $s+\eps<1$. We now use that $\gamma \ge -4s \ge -4$ and $0 \le \gamma +4 s +2\eps \le 2$ to get,}
  & \le  C_\Psi C_\varphi^2 C_b  \int_{\R^d \times \R^d} f(v)f(w)  (1+|v-w|^{2})   \dv \dw \\
  & \le  2 C_\Psi C_b  (M_0 + E_0)^2 (4 \|D^2_v \varphi\|_{L^\infty} + 4\|\varphi\|_{L^\infty})^2
\end{align*}
with $C_b = \int_{\Sd} (1-\cos \theta)^{s+\eps} b (\cos \theta) \dsigma$. 
\end{proof}

\subsection{Bounded kernel functions}

We recall that the collision operator is defined by,
\[ Q_B(f,f)(v) = \iint_{\Sd \times \R^d} (f(w')f(v') -f (w) f(v) ) B(w-v,\sigma) \dsigma \dw  \]
with
\[ v' = \frac{v+w}2 + \frac{|v-w|}2 \sigma \quad \text{ and } w' = \frac{v+w}2 - \frac{|v-w|}2 \sigma \]
and a function $B \colon \R^d \times \Sd \to [0,+\infty)$. 
As explained in  \cite{MR0339665}, it is easy to construct solutions when $B$ is bounded.
We have to justify that we can get the family of inequalities contained in the definition
of suitable weak solutions. 
\begin{lem}[Existence of solutions for bounded kernel functions and artificial viscosity]\label{l:existence-approx}
  Let $B \colon \R^d \times \Sd \to [0,+\infty)$ be bounded and $\fin \in L^1 (\R^d)$.
  For all $\eps >0$, there exists $f^\eps \in C((0,+\infty), L^1(\R^d))$ with $f^\eps \ge 0$ a.e. such that 
  \[ \partial_t f^\eps = Q_{B} (f^\eps,f^\eps) + \eps \Delta f^\eps \quad \text{ in } \mathcal{D}'((0,+\infty) \times \R^d) \]
  and $f^\eps(0,v) = \fin (v)$.

  If moreover $\fin \ge 0$ and $\fin \ln \fin \in L^1(\R^d)$, then $f^\eps \ge 0$ a.e. in $(0,+\infty) \times \R^d$ and
  $f^\eps(t,\cdot) \ln f^\eps (t,\cdot) \in L^1 (\R^d)$ and $Q_B (f^\eps,f^\eps) \ln f^\eps \in L^1((0,+\infty) \times \R^d)$ and
  \[
    \frac{\mathrm{d}}{\dt} \int_{\R^d} f^\eps \ln f^\eps (t,v) \dv + \int_{\R^d} Q_B (f^\eps,f^\eps) \ln f^\eps \dv + \eps  \int_{\R^d} \frac{|\nabla_v f^\eps|^2}{f^\eps} (t,v) \dv  \le 0
    \quad \text{ in } \mathcal{D}' ((0,+\infty))
  \]
    where the entropy dissipation term is understood in the sense of \eqref{e:dissip-weak}.
\end{lem}
\begin{proof}
  We split the proof in several steps. We first write the collision operator into gain and loss operators.
  We then construct solutions by an iterative scheme. We finally pass to the limit thanks to the monotone
  convergence theorem.
  
  \paragraph{Splitting the collision operator.}
  We  consider $m = \int_{\R^d} \fin (v) \dv$. We then split $Q_B$ into the gain part and the loss part,
  \( Q_B (f,f) = Q_B^+ (f,f) - Q_B^- (f,f) \)
  with
  \begin{align*}
    Q_B^+ (f,f)(v) & = \iint_{\Sd \times \R^d} f(w')f(v') B(w-v,\sigma) \dsigma \dw  \\
    Q_B^- (f,f)(v) & = \iint_{\Sd \times \R^d} f (w) f(v)  B(w-v,\sigma) \dsigma \dw  \\
                   & = f(v) \bigg( f \ast_v \bar B \bigg)
  \end{align*}
  where
  \( \bar B (v) = \int_{\Sd} B(v,\sigma) \dd \sigma.\)
  Since $B$ is bounded, so is $\bar B$. Moreover, $Q_B^+ (f,f)$ is well-defined and essentially bounded
  for $f \in L^1 (\R^d)$. Since $\int_{\R^d} Q_B (f,f) (v) \dv =0$, we have in particular,
  \[ \int_{\R^d} Q_B^+ (f,f) (v) \dv = \int_{\R^d} Q_B^- (f,f) (v) \dv.\]

  Let $\mathcal{B} >0$ such that for all $z,\sigma$,
  we have $0 \le \bar B(z,\sigma) \le \mathcal{B}$. We then consider the equation
  \[ \partial_t f + \mathcal{B} m f = Q_B^+ (f,f) + f ((\mathcal{B}-\bar B) \ast_v f) := \bar Q (f,f).\]
  We remark that $\bar Q (f,f)$ is well-defined for $f \in L^1 (\R^d)$ and $\bar Q (f,f) \in L^1 \cap L^\infty (\R^d)$ and
  \[ \int_{\R^d} \bar Q (f,f) (v) \dv = \mathcal{B} \left( \int_{\R^d} f(v) \dv \right)^2. \]
  It is also monotone increasing in $f$.

  \paragraph{The iterative scheme.}
  We then construct a solution  of the previous equation by the following iterative scheme, 
  \[ \forall n \ge 1, \qquad \partial_t f^n + \mathcal{B} m f^n = \bar Q (f^{n-1},f^{n-1}) + \eps \Delta f^n\]
  with $f^0 = e^{-\mathcal{B} mt+\eps t \Delta} \fin$. More precisely, we simply define for $f^{n-1} \in C((0,+\infty), L^1 (\R^d))$, 
  \[ \forall n \ge 1, \qquad f^n (t,v) = e^{-\mathcal{B} m t+\eps t \Delta} \fin (v) + \int_0^t e^{-\mathcal{B} m (t-s)+ \eps (t-s)\Delta}  \bar Q (f^{n-1},f^{n-1}) (s,v) \dd s.\]
  We have in particular $f^n \in C ((0,+\infty),L^1 (\R^d))$. 
  We also remark that $f^0 \in C((0,+\infty),L^1(\R^d))$ and $f^0 \le f^1$. Moreover, if $f^{n-1} \in C((0,+\infty),L^1 (\R^d))$ and $f^{n-2} \le f^{n-1}$
  then we have,
  \begin{align*}
    f^n (t,v) &\ge e^{-\mathcal{B} m t+\eps t \Delta} \fin (v) + \int_0^t e^{-\mathcal{B} m (t-s)+\eps (t-s)\Delta}  \bar Q (f^{n-2},f^{n-2}) (s,v) \dv = f^{n-1}(t,v),\\
    \int_{\R^d} f^n (t,v) \dv &= e^{-\mathcal{B} m t} m + \mathcal{B} \int_0^t e^{-\mathcal{B} m (t-s)} \left( \int_{\R^d} f^{n-1}(t,v) \dv \right)^2 \dd s.
  \end{align*}
  The second equality implies that $\mu_n (t) = \int_{\R^d} f^n (t,v) \dv$ satisfies the ODE
  \[ \dot \mu = \mathcal{B} \mu (\mu -m).\]
  Since $\mu_n(0) = m$, we conclude that $\mu_n (t) =m$ for $t>0$, that is to say,
  \[ \int_{\R^d } f^n (t,v) \dv = m.\]

\paragraph{Passing to the limit.}
If $f^\eps(t,v)$ denotes $\lim_{n \to \infty} f^n (t,v)$. For all $t>0$, Beppo-Levi's theorem implies that $f^\eps (t,\cdot) \in L^1 (\R^d)$.
Moreover, $f^\eps (t,v) = \sup_{n \ge 0} f^n (t,v)$ and in particular, $f^\eps \colon (0,+\infty) \to L^1 (\R^d)$ is lower semi-continuous.
We also have, for all $t > 0$,
  \[\int_{\R^d} f^\eps(t,v) \dv = m.\]
In particular, $f^\eps \in L^\infty ((0,+\infty), L^1 (\R^d))$. 
  
  Remark the sequence $\{ \bar Q (f^n,f^n) \}_n$ is also monotone increasing and converges towards $\bar Q (f^\eps,f^\eps)$
  (again by Beppo-Levi's theorem for $t >0$ fixed). We conclude that, at the limit, we have for all $t>0$ the following equality in
  $L^\infty ((0,+\infty),L^1(\R^d))$, 
  \[ f^\eps (t,v) = e^{-\mathcal{B} m t+\eps t \Delta} \fin (v) + \int_0^t e^{-\mathcal{B} m (t-s)+\eps (t-s) \Delta}  \bar Q (f^\eps,f^\eps) (s,v) \dd s.\]
  In particular, $f^\eps \in C((0,+\infty), L^1 (\R^d))$ and it satisfies,
  \[ \partial_t f^\eps + \mathcal{B} m f^\eps = \bar Q (f^\eps,f^\eps) + \eps \Delta f^\eps \quad \text{ in } \mathcal{D}'((0,+\infty) \times \R^d) .\]
  Conservation of mass also implies that $\bar Q (f^\eps,f^\eps) = Q_B (f^\eps,f^\eps) + \mathcal{B} m f^\eps$ and we finally get,
  \[ \partial_t f^\eps  = Q_B (f^\eps,f^\eps) + \eps \Delta f^\eps\quad \text{ in } \mathcal{D}'((0,+\infty) \times \R^d). \]

  \paragraph{Entropy dissipation.}
  If moreover $\fin$ is such that $\fin \ge 0$, we immediately get $f^\eps \ge 0$ from the representation formula we used above.
  Moreover, if $\fin \ln \fin \in L^1 (\R^d)$, then we can use a standard approximation procedure to get,
\[
    \frac{\mathrm{d}}{\dt} \int_{\R^d} f^\eps \ln f^\eps (t,v) \dv + \int_{\R^d} Q_B (f^\eps,f^\eps) \ln f^\eps \dv
    + \eps  \int_{\R^d} \frac{|\nabla_v f^\eps|^2}{f^\eps} (t,v) \dv  \le 0
    \quad \text{ in } \mathcal{D}' ((0,+\infty))
  \]
  where $\int_{\R^d} Q_B (f^\eps,f^\eps) \ln f \dv$ is understood in the sense of \eqref{e:dissip-weak}. 
\end{proof}

\subsection{Approximate kernels}

\begin{lem}[Existence of solutions for non-cut off kernel functions and artificial viscosity]\label{l:existence-approx-2}
  Let the kernel function $B$ be of non-cut off type, that is to say
  $B(z,\cos \theta) = |z|^\gamma b (\cos \theta)$ with $\gamma > -d$ and $b$ satisfying \eqref{e:b} for some $s \in (0,1)$. 
  For all $\eps >0$, there exists
  \[f^\eps \in L^\infty((0,+\infty), L^1_2(\R^d)) \cap L^1 ((0,+\infty), W^{1,1} (\R^d))\]
  with $f^\eps \ge 0$ a.e. such that 
  \[ \partial_t f^\eps = Q_{B} (f^\eps,f^\eps) + \eps \Delta f^\eps \quad \text{ in } \mathcal{D}'((0,+\infty) \times \R^d) \]
  and $f^\eps(0,v) = \fin (v)$. The collision operator $Q_B (f,f)$ is understood in the sense of Lemma~\ref{l:operator-weak}.

  Moreover, 
  \[ \frac{\dd}{\dt} \int_{\R^d} f^\eps \ln f^\eps (t,v) \dv - \int_{\R^d} Q_B (f^\eps,f^\eps) \ln f^\eps (t,v) \dv \le 0 \text{ in } \mathcal{D}'((0,+\infty))\]
  where the entropy dissipation is understood in the following sense,
  \begin{multline*}
    -\int_{\R^d} Q_B (f^\eps,f^\eps) \ln f^\eps (v) \dv = \\\frac12 \int_{\R^d \times \R^d \times \Sd}
      (f^\eps(v')f^\eps(w') -f^\eps(v)f^\eps(w)) \ln \frac{f^\eps(v')f^\eps(w')}{f^\eps(v)f^\eps(w)} B(|v-w|,\cos \theta) \dv \dw \dd \sigma
     \ge 0.
   \end{multline*}
   
  Finally, for any $a>0$, we have
  \[ \frac{\mathrm{d}}{\mathrm{d}t} \int_{\R^d} \varphi_a (f^\eps(t,v)) \dv \le  - \int_{\R^d \times \R^d} d_{\varphi_a} (f^\eps(v),f^\eps(v')) K(v,v') \dv' + c_c \int_{\R^d} \Phi_a (f^\eps) (f^\eps \ast_v |\cdot|^\gamma) \dv \]
  where $\varphi_a (r) = (r-a)_+$ and
  the kernel $K$ is defined by \eqref{e:kernel} and $c_c$ is the positive constant from the cancellation lemma. 
\end{lem}
\begin{proof}
We consider approximate kernel functions that are bounded and we then pass to the limit. 
  
\paragraph{Approximate kernel functions.}
We  consider bounded functions $\{B_n\}_{n}$ approximating the original function $B$ in the non-cut off case, 
\[B_n(z,\cos \theta) = X(n|z|)|z|^\gamma b (\cos \theta) \un_{|\sin (\theta /2)| \ge 1/n}\]
for some $C^\infty$ function $X \colon [0,+\infty) \to [0,1]$ supported in $[1/2,+\infty)$, $X=1$ in $[1,+\infty)$ and $0 \le X' \le 3$. 
The function $B_n$ is bounded since $X (n|z|) >0$ implies that $|z| \ge 1/(2n)$, ensuring that $|z|^\gamma$ is bounded from above.
Moreover,
\[ \forall n \ge 1, B_n (z,\cos \theta) \le B_{n+1} (z,\cos \theta).\]

\paragraph{Approximate lower order term.}
As far as $\mathcal{R}_n$ is concerned (see Lemma~\ref{l:cancel}),
we use the special form of $B_n$ in order to write, with $\cos(\theta/2) = \left(\frac{1-\sigma \cdot e}2\right)^{\frac12}$,
\begin{align*}
  \mathcal{R}_n (z)  = &
     |z|^\gamma \int_{\Sd} \left(\cos(\theta/2)^{-d-\gamma} X \left(\frac{n |z|}{|\cos(\theta/2)|}\right) - X(n|z|)\right)  b (\cos \theta) \un_{|\sin (\theta /2)| \ge 1/n} \dsigma \\
                    = & |z|^\gamma X (n|z|) \int_{\Sd} \left(\cos(\theta/2)^{-d-\gamma} -1 \right)  b (\cos \theta) \un_{|\sin (\theta /2)| \ge 1/n} \dsigma \\
  + &  |z|^\gamma \int_{\Sd} \cos(\theta/2)^{-d-\gamma} \left( X \left(\frac{n |z|}{|\cos(\theta/2)|}\right) - X(n|z|)\right)  b (\cos \theta) \un_{|\sin (\theta /2)| \ge 1/n} \dsigma \\
= & \mathcal{R}_n^1 (z) + \mathcal{R}_n^2 (z).
\end{align*}
Since $0 \le X \le 1$, we immediately see that
\begin{equation}\label{e:r1} \mathcal{R}_n^1 (z) \le c_c |z|^\gamma 
\end{equation}
with $c_c = \int_{\Sd} \left(\cos(\theta/2)^{-d-\gamma} -1 \right)  b (\cos \theta) \dsigma < +\infty$ (see Lemma~\ref{l:cancel}). 
As far as $\mathcal{R}_n^2$ is concerned, we use the fact that $X'$ is supported in $[1/2,1]$ and $0 \le X' \le 3$ in order to write
\[ \mathcal{R}_n^2 (z) = 0 \text{ if } |z| \ge 1/n \]
and for $|z| \le 1/n$,
\begin{align*}
  \mathcal{R}_n^2 (z) & \le 3 |z|^\gamma  \int_{\Sd} \cos(\theta/2)^{-d-\gamma} \left(\frac{1}{|\cos(\theta/2)|} - 1\right)   b (\cos \theta) \un_{|\sin (\theta /2)| \ge 1/n} \dsigma \\
& \le \bar c_c |z|^\gamma
\end{align*}
with $\bar c_c = 3 \int_{\Sd} \cos(\theta/2)^{-d-\gamma} \left(\frac{1}{|\cos(\theta/2)|} - 1\right)   b (\cos \theta)  \dsigma$. 
We thus proved that
\begin{equation}
  \label{e:r2}
  \mathcal{R}_n^2 (z) \le \bar c_c |z|^\gamma \un_{\{ |z| \le 1/n\}}.
\end{equation}

\paragraph{Compactness and strong convergence.}
Let $f^n$ be given by Lemma~\ref{l:existence-approx} where $B= B_n$. In particular, we write the entropy dissipation inequality,
\begin{equation}
  \label{e:entropy-dissip}
    \frac{\mathrm{d}}{\dt} \int_{\R^d} f^n \ln f^n (t,v) \dv + \int_{\R^d} Q_{B_n} (f^n,f^n) \ln f^n \dv + \eps  \int_{\R^d} \frac{|\nabla_v f^n|^2}{f^n} (t,v) \dv  \le 0    \quad \text{ in } \mathcal{D}' ((0,+\infty)).
\end{equation}
It ensures that $\nabla_v \sqrt{f^n}$ is bounded in $L^2 ((0,+\infty) \times (\R^d))$. In turn, this implies that $\nabla_v f^n$ is bounded in $L^1((0,T) \times \R^d)$ and
\[ f^n \text{ is bounded in } L^1((0,T), W^{1,1} (\R^d)) \text{ for all } T>0. \]
In particular,
\( f^n \text{ is bounded in } L^1((0,T),L^{\frac{d}{d-2}} (\R^d) )\text{ for all } T>0. \) Since this sequence is also bounded in $L^\infty ((0,T),L^1(\R^d))$, we conclude by interpolation that for all $\iota >0$,
\begin{equation}
  \label{e:piota}
    f^n \text{ is bounded in } L^{1+\iota} \left((0,T), L^{p_\iota} (\R^d)\right) \text{ with } p_\iota = \frac{d(1+\iota)}{d(1+\iota)-2}.
\end{equation}

Recalling that $B_n (z, \cos \theta) = \Psi_n (|z|) b_n (\cos \theta)$,  we  have 
for $\mathcal{F} (t,v,w) = \sqrt{ f(t,v) f(t,w) \Psi_n (|v-w|)}$, 
\begin{align*}
  &\iiint_{(0,+\infty) \times \R^d \times \R^d \times \Sd} \left(\mathcal{F}(t,v',w') - \mathcal{F}(t,v,w) \right)^2 b_n (\cos \theta) \dt \dv \dw \dsigma\\
  & \le \int_{(0,+\infty) \times \R^d} Q_{B_n} (f^n,f^n) \ln f^n \dt \dv  \\
  & \le \int_{\R^d} \fin \ln \fin (v) \dv \\
  & \le H_0
\end{align*}
thanks to Lemma~\ref{l:dissip} and \eqref{e:hydro}.

Moreover, for $\psi \in \mathcal{D} ((0,+\infty) \times \R^d)$,
\[
  \int_{(0,+\infty) \times \R^d} f^n (t,v) \partial_t \psi (t,v) \dt \dv =  \eps \int_{(0,+\infty) \times \R^d} \nabla_v f^n \cdot \nabla_v \psi \dt \dv - \int_{(0,+\infty) \times \R^d} Q_{B_n} (f^n,f^n) \psi \dt \dv
\]
where we use Lemma~\ref{l:operator-weak} to define the last term:
if $\delta \psi (t,v,w,v',w') = \psi (t,v)+\psi (t,w) - \psi (t,v') - \psi (t,w')$, 
we have
\begin{multline*}
  \int_{(0,+\infty) \times \R^d} Q_{B_n} (f^n,f^n) \psi \dt \dv \\
 = \frac14 \int_{(0,+\infty) \times \R^d \times \R^d \times \Sd} (\mathcal{F}(t,v',w') - \mathcal{F}(t,v,w))\sqrt{b_n (\cos \theta)} \\
  \times  (\mathcal{F}(t,v,w) + \mathcal{F}(t,v',w')) \delta \psi (t,v,w,v',w')  \sqrt{b_n (\cos \theta)} \dt \dv \dw \dsigma .
\end{multline*}
In particular, if $\psi$ is supported in $(0,T) \times \R^d$, Lemma~\ref{l:operator-weak} implies that
\[
  \left|\int_{(0,+\infty) \times \R^d} Q_{B_n} (f^n,f^n) \psi \dt \dv \right|^2  \le  4^3 C_b H_0    (M_0 + E_0)^2 ( \|D^2_v \psi\|_{L^\infty} + \|\psi\|_{L^\infty})^2 T
\]
where $E_0,M_0$ and $H_0$ appear in \eqref{e:hydro}. 
We thus get,
  \begin{multline*}
    \left| \int_{(0,+\infty) \times \R^d} f^n (t,v) \partial_t \psi (t,v) \dt \dv \right| \le \frac12 \eps \sqrt{T} \sqrt{M_0} H_0 \|\nabla_v \psi \|_{L^\infty ((0,+\infty)\times \R^d)}  \\
     + 2^3 \sqrt{T} \sqrt{C_b H_0}    (M_0 + E_0) ( \|D^2_v \psi\|_{L^\infty} + \|\psi\|_{L^\infty}).
  \end{multline*}
  In particular,
  \begin{equation}
    \label{e:dtf-bound}
    \partial_t f^n \text{  is bounded in } L^1 ((0,T),W^{-2,1} (\R^d)) \text{ for all } T >0.
  \end{equation}
  
Aubin's lemma then implies that $f^n$ is relatively compact in $L^1 ((0,T), L^1 (B_R))$ for all $T>0$ and $R>0$.
In particular, we can extract a subsequence $n_j$ such that $f^{n_j}$ converges almost everywhere towards a function
$f \in L^1_{\mathrm{loc}} ((0,+\infty) \times \R^d)$. In the remainder of the proof,
we drop the subscript $j$ corresponding to the subsequence for clarity. 

\paragraph{Second moments.}
For bounded kernels $K_n$, we can use the fact that $\int_{\R^d} Q_{B_n} (f,f) \dv =0$ in order to get
that
\[
  \frac{\dd}{\dt} \int_{\R^d} f^n (t,v) |v|^2 \dv \le \eps \int_{\R^d} f^n (t,v) \Delta (|v|^2)
  \le 2d\eps M_0
\]
where $M_0$ is the initial mass, see \eqref{e:hydro}. 
In particular,  we have for a.e. $t >0$,
\begin{equation}
  \label{e:energy}
 \int_{\R^d} f^n (t,v) |v|^2 \dv \le 2d \eps M_0  + E_0
\end{equation}
where $E_0$ comes from \eqref{e:hydro}. 

\paragraph{The $\varphi_a$-inequalities.}
We can repeat the computation presented in the introduction and apply it when $B=B_n$ and we get,
\begin{multline*}
  \frac{\mathrm{d}}{\mathrm{d}t} \int_{\R^d} \varphi_a (f^n(t,v)) \dv \\
  \le  - \int_{\R^d \times \R^d} d_{\varphi_a} (f^n(v),f^n(v')) K_n(v,v') \dv \dv' +  \int_{\R^d} \Phi_a (f^n) (f^n \ast_v \mathcal{R}_n) \dv \text{ in } \mathcal{D}'((0,+\infty))
\end{multline*}
where the kernel $K_n$ is defined by \eqref{e:kernel} and $c_c$ is the positive constant from the cancellation lemma.

We can use Fatou's lemma in order to pass to the infimum limit in the dissipation term: for all non-negative $\psi \in \mathcal{D}((0,+\infty))$, 
\begin{multline*}
  \int_{\R} \int_{\R^d \times \R^d} d_{\varphi_a} (f(v),f(v')) K(v,v') \dv \dv' \psi (t) \dt \\
  \le \liminf_{n \to \infty} \int_{\R} \int_{\R^d \times \R^d} d_{\varphi_a} (f^n(t,v),f^n(t,v')) K_n(t,v,v') \dv \dv' \psi (t) \dt. 
\end{multline*}

Thanks to almost everywhere convergence of $f^n$ in $(0,+\infty) \times \R^d$, the entropy estimate~\eqref{e:entropy-dissip} and second moment estimate~\eqref{e:energy}, we conclude that
\[ \int_{\R^d} \varphi_a(f^n(t,v)) \dv \to \int_{\R^d} \varphi_a(f(t,v)) \dv \text{ in } \mathcal{D}'((0,+\infty)).\]

As far as the source term is concerned, we write for a non-negative $\psi \in \mathcal{D}((0,+\infty))$, 
\[ \int_{\R} \int_{\R^d} \Phi_a (f^n) (f^n \ast_v \mathcal{R}_n) \dv \psi (t) \dt \le C_{\mathcal{R}} \int_{\R} \int_{\R^d \times \R^d} \Phi_a (f^n (t,v)) f^n (t,w) |v-w|^\gamma  \psi (t) \dv \dw \dt.\]
The integrand converges almost everywhere towards $\Phi_a (f(t,v)) f(t,w)  |v-w|^\gamma \psi (t)$.
We now prove that it is uniformly integrable. In order to do so, we consider $\iota >0$ small and we write,
\begin{align*}
  \int_{\R} \int_{\R^d \times \R^d} \Phi_a^{1+\iota} (f^n (t,v))
  &(f^n)^{1+\iota} (t,w) |v-w|^{\gamma(1+\iota)}  \psi^{1+\iota} (t) \dv \dw \dt \\
  & = \int_{\R} \int_{\R^d} \Phi_a^{1+\iota} (f^n (t,v)) (f^n)^{1+\iota} \ast_v |\cdot|^{\gamma(1+\iota)}  \psi (t)^{1+\iota} \dv  \dt\\
  & \le C_{\iota,\gamma,d}  \int_{\R} \| \Phi_a^{1+\iota}\|_{L^{r_i'}(\R^d)}\|(f^n)^{1+\iota}\|_{L^{\frac{p_\iota}{1+\iota}}(\R^d)}   \psi (t)^{1+\iota}   \dt
\end{align*}
with $\frac{1}{r_\iota'} + \frac{1}{r_\iota}=1$ and $1 + \frac{1}{r_\iota} = \frac{1+\iota}{p_\iota} - \frac{\gamma (1+\iota)}{d}$.
Since $\Phi_a = a \un_{\{ f > a\}}$, we have
\[ \| \Phi_a^{1+\iota}\|_{L^{r_\iota'}(\R^d)} = a^{1+\iota} |\{ f > a \}|^{\frac1{r_\iota'}} \le a^{1+\iota - \frac{1}{r_\iota'}} M_0
\]
with $M_0$ from \eqref{e:hydro}. We use this estimate to write,
\begin{align*}
  \int_{\R} \int_{\R^d \times \R^d} \Phi_a^{1+\iota} (f^n (t,v))
  &(f^n)^{1+\iota} (t,w) |v-w|^{\gamma(1+\iota)}  \psi^{1+\iota} (t) \dv \dw \dt \\
  & \le C_{\iota,\gamma,d} a^{1+\iota - \frac{1}{r_\iota'}} M_0 \int_{\R} \|f^n\|^{1+\iota}_{L^{p_\iota}(\R^d)}   \psi (t)^{1+\iota}   \dt.
\end{align*}
We now use \eqref{e:piota} to conclude that the sequence is uniformly integrable. 

We next prove that this sequence  is tight.
\begin{align*}
  & \int_{\R} \int_{\R^d \times \R^d} \Phi_a (f^n (t,v)) (f^n) (t,w) |v-w|^{\gamma} (|v| +|w|)^\iota \psi (t) \dv \dw \dt \\
  & \le 2^\iota \int_{\R} \int_{\R^d \times \R^d} \Phi_a (f^n (t,v)) (f^n) (t,w) |v-w|^{\gamma} (|v|^\iota +|v-w|^\iota) \psi (t) \dv \dw \dt \\
  & =  2^\iota \int_{\R} \int_{\R^d \times \R^d} \Phi_a (f^n (t,v)) \bigg( ((f^n) \ast_v |\cdot|^{\gamma}) |v|^\iota + f^n \ast_v |\cdot|^{\gamma+\iota} \bigg) \psi (t) \dv \dw \dt \\
& = (I) + (II).
\end{align*}
The second term can be treated as above. As far as the first term is concerned, we can use that $f^n$ has finite second moment to handle it.

\paragraph{Weak solution of the equation.}
It is also possible to use the piece of information above to pass to the limit in the weak
formulation of the equation. 
\end{proof}

\subsection{Vanishing viscosity}

We now prove Proposition~\ref{p:existence} by considering a vanishing viscosity $\eps_n$.
\begin{proof}[Proof of Proposition~\ref{p:existence}]
Let $f^n$ be the solution of the viscous Boltzmann equation given by Lemma~\ref{l:existence-approx-2}
associated with a viscosity $\eps_n = 1/n$.

\paragraph{Compactness and strong convergence.}
We use \cite[Proposition~0.2]{chaker2022entropy} in order to get the following lower bound on the entropy dissipation estimate,
\begin{align*}
  D(f^n) & = \int_{\R^d} Q (f^n,f^n) \ln f^n \dv \\
  & \ge c \int_{\R^d \times \R^d} \frac{(\sqrt{f^n} (v') - \sqrt{f^n} (v))^2}{d_{\mathrm{GS}} (v,v')^{d+2s}} (\langle v \rangle \langle v' \rangle)^{\frac{\gamma+2s+1}{2}} \un_{\{ d_{\mathrm{GS}} (v,v') \le 1\}}\dv \dv' - C M_0^2
\end{align*}
where $c,C$ only depend on $d,\gamma,s$ and $m_0,M_0,E_0,H_0$ from \eqref{e:hydro}. We recall that
\[ d_{\mathrm{GS}} (v,v') = \sqrt{|v-v'|^2 + \frac14 (|v|^2 -|v'|^2)^2}.\]
In particular for $v,v' \in B_R$, we have
\[ d_{\mathrm{GS}} (v,v') \le \langle R \rangle  |v-v'|\]
where we abuse notation by writing $\langle R \rangle = \sqrt{1+R^2}$. 
In particular,
\[ D (f^n) \ge  c \int_{B_R \times B_R} \frac{(\sqrt{f^n} (v') - \sqrt{f^n} (v))^2}{|v-v'|^{d+2s}} \langle R \rangle^{\gamma-d+1} \un_{\{ |v-v'| \le \langle R \rangle^{-1}\}}\dv \dv' - C M_0^2 .\]
This implies that
\[ \sqrt{f^n} \text{ is bounded in } L^2 ((0,T), H^s (B_R)) \text{ for all } T,R>0.\]
In particular, for $\iota < s$,
\[ \frac{\sqrt{f^n}(t,v) - \sqrt{f^n} (t,v')}{|v-v'|^{s+ \iota}} \text{ is bounded in } L^2 \left((0,T) \times B_\R \times B_R, \dt \frac{\dv \dv'}{|v-v'|^{d-2 \iota}} \right) .\]
We deduce from conservation of mass that
\[ \sqrt{f^n}(t,v) - \sqrt{f^n}(t,v') \text{ is bounded in } L^2 \left((0,T) \times B_R \times B_R,  \dt \frac{\dv \dv'}{|v-v'|^{d-2 \iota}} \right).\]
We conclude that
\[ \frac{f^n(t,v) - f^n (t,v')}{|v-v'|^{d+(s-\iota)}} \text{ is bounded in } L^1 \left((0,T) \times B_\R \times B_R, \dt \dv \dv'\right) .\]
In other words, we get on the one hand,
\[ f^n \text{ is bounded in } L^1 ((0,T),W^{1,s-\iota}(B_R)) \text{ for all } T,R>0.\]
On the other hand, we can argue as in the proof of Lemma~\ref{l:existence-approx-2}, see \eqref{e:dtf-bound} and get that
\[ \partial_t f^n \text{ is bounded in } L^1((0,T), W^{-2,1} (B_R)) \text{ for all } T,R>0.\]
We conclude that $f^n$ is compact in $L^1((0,T) \times B_R)$, in particular there exists a subsequence that converges in
$L^1$ and almost everywhere in $(0,+\infty) \times \R^d$ towards a non-negative function $f$. 

\paragraph{Uniform integrability.}
We can use \cite[Theorem~0.1]{chaker2022entropy} and get 
\begin{align*}
  D(f^n) & = \int_{\R^d} Q (f^n,f^n) \ln f^n \dv \\
  & \ge c \| f^n\|_{L^{p_0}_{k_0} (\R^d)} - C M_0^2
\end{align*}
with $\frac1{k_0} = 1- \frac{2s}d$ and $k_0= \gamma +2s -\frac{2s}d<0$. The constants
$c,C$ only depending on $d,s,\gamma$ and bounds from \eqref{e:hydro}.
In particular,
\[ f^n \text{ is bounded in } L^1 ((0,T), L^{p_0}_{k_0} (\R^d)).\]

\paragraph{Tightness.}
With almost everywhere convergence, we can also ensure that the function $f$ has a second moment
for almost every time $t>0$,
\[ \int_{\R^d} f (t,v) |v|^2 \dv \le E_0.\]
(it is enough to use Fatou's lemma in \eqref{e:energy}). 

\paragraph{Passing to the limit.}
Strong convergence together with uniform integrability and the second moment estimate  allows us to
pass to the limit in the family of inequalities associated with convex functions $\varphi_a$ for $a>0$.

It is also possible to pass to the limit in the weak formulation of equation and get
  weak solutions of the Boltzmann equation.
\end{proof}

\bibliographystyle{plain}
\bibliography{very-soft}

\end{document}